\documentclass[11pt,reqno]{article}
\usepackage{amsmath,mathrsfs,amssymb,amsthm,amscd,bbm}
\usepackage{color,epsfig,latexsym,graphicx,eucal,layout,fancyhdr}
\usepackage{enumerate,hyperref}
\usepackage[]{fontenc}
\usepackage[all]{xy}
\usepackage{hyperref}
\usepackage{graphics}
\usepackage{a4wide,verbatim}
\usepackage{psfrag}

\newcounter{EQNR}
\renewcommand{\contentsname}{Contents\\
{\footnotesize\normalfont(A table of contents should normally not be included)}}

\newtheorem{theorem}{Theorem}

\newtheorem{corollary}{Corollary}

\newtheorem{proposition}{Proposition}

\DeclareMathOperator{\PSL}{PSL}
\DeclareMathOperator{\hyp}{hyp}
\let\Im\relax
\DeclareMathOperator{\Im}{Im}
\let\Re\relax
\DeclareMathOperator{\Re}{Re}
\newcommand{\C}{\mathbb{C}}
\newcommand{\h}{\mathbb{H}}
\newcommand{\N}{\mathbb{N}}

\newcommand{\R}{\mathbb{R}}
\newcommand{\Z}{\mathbb{Z}}
\newcommand{\E}{\mathcal{E}}

\newcommand{\abs}[1]{\left\vert#1\right\vert}

\DeclareMathOperator{\vol}{vol}

\def\ddu2{{\frac{\partial^2}{\partial u^2}}}
\def\Re{{\mathrm{Re}}}
\def\Im{{\mathrm{Im}}}
\def\ZZ {{\mathbb Z}}

\def\RR {{\mathbb R}}
\def\CC {{\mathbb C}}

\def\HH {{\mathbb H}}

\begin{document}

\title{Kronecker limit functions and an extension of the Rohrlich-Jensen formula}
\author{James Cogdell \and Jay Jorgenson
\footnote{The second named author acknowledges grant support from several PSC-CUNY Awards, which are jointly funded
by the Professional Staff Congress and The City University of New York.}\and Lejla Smajlovi\'{c}}
\maketitle

\begin{abstract}\noindent
In \cite{Ro84} Rohrlich proved a modular analogue of Jensen's formula.  Under certain conditions, the Rohrlich-Jensen
formula expresses an integral of the log-norm $\log \Vert f \Vert$ of a $\text{\rm PSL}(2,\ZZ)$ modular form $f$ in terms
of the Dedekind Delta function evaluated at the divisor of $f$.  In \cite{BK19} the authors re-interpreted the Rohrlich-Jensen
formula as evaluating a regularized inner product of $\log \Vert f \Vert$ and extended the result to compute a regularized
inner product of $\log \Vert f \Vert$ with what amounts to powers of the Hauptmoduli of $\text{\rm PSL}(2,\ZZ)$.  In the present article, we
revisit the Rohrlich-Jensen formula and prove that it can be viewed as a regularized inner product of special values of two
Poincar\'e series, one of which is the Niebur-Poincar\'e series and the other is the resolvent kernel of the Laplacian.
The regularized inner product can be seen as a type of Maass-Selberg relation.
In this form, we develop a Rohrlich-Jensen formula associated to any Fuchsian group $\Gamma$ of the first kind  with one cusp by employing a type of Kronecker limit formula associated to the resolvent kernel.  We
present two examples of our main result: First,  when $\Gamma$ is the full modular group $\text{\rm PSL}(2,\ZZ)$, thus reproving the theorems
from \cite{BK19}; and second when $\Gamma$ is an Atkin-Lehner group $\Gamma_{0}(N)^+$, where explicit computations are given for certain
genus zero, one and two levels.
\end{abstract}

\vskip .15in
\section{Introduction and statement of results}

\subsection{The Poisson-Jensen formula}

Let $D_{R} = \{z =x+iy\in \CC : \vert z \vert < R\}$ be the disc of radius $R$ centered at the
origin in the complex plane $\CC$.  Let $F$ be a non-constant meromorphic function on the closure
 $\overline{D_{R}}$ of $D_{R}$.  Denote by $c_{F}$ the leading non-zero 
 coefficient of $F$ at zero, meaning that
 for some integer $m$ we have that
 $F(z) = c_{F}z^{m} + O(z^{m+1})$ as $z$ approaches zero.  For any $a \in D_{R}$, let $n_{F}(a)$ denote
 the order of $F$ at $a$; there are a finite number of points $a$ for which $n_{F}(a) \neq 0$.
 With this, Jensen's formula, as stated on page 341 of \cite{La99}, asserts that
\begin{equation}\label{Jensen}
\frac{1}{2\pi}\int\limits_{0}^{2\pi}\log\vert F(Re^{i\theta})\vert d\theta
+ \sum\limits_{a \in D_{R}} n_{F}(a) \log (\vert a \vert/R) + n_{F}(0) \log (1/R) = \log \vert c_{F}\vert.
\end{equation}
One can consider the action of a M\"obius transformation which preserves $D_{R}$ and seek to determine the
resulting expression from \eqref{Jensen}.  Such a consideration leads to the
Poisson-Jensen formula, and we refer the reader to page 161 of \cite{La87} for a statement and proof.

On their own, the Jensen formula and the Poisson-Jensen formula paved the way toward Nevanlinna theory, which in
its most elementary interpretation establishes subtle growth estimates for meromorphic functions; see Chapter VI
of \cite{La99}.  Going further, Nevanlinna theory provided motivation for Vojta's conjectures whose insight
into arithmetic algebraic geometry is profound.  In particular, page 34 of \cite{Vo87} contains a type
of ``dictionary'' which translates between Nevalinna theory and number theory
where Vojta asserts that Jensen's formula should be viewed as analogous to the Artin-Whaples product formula
from class field theory.

\subsection{A modular generalization}

In \cite{Ro84} Rohrlich proved what he aptly called a modular version of Jensen's formula.  We now shall
describe Rohrlich's result.

Let $f$ be a meromorphic function on the upper half plane $\HH$ which is invariant with respect to the
action of the full modular group $\mathrm{PSL}(2,\Z)$.  Set $\mathcal{F}$ to be the ``usual'' fundamental domain
of the quotient $\mathrm{PSL}(2,\Z)\backslash \HH$, and let $d\mu$ denote the area form of the hyperbolic metric.
Assume that $f$  does not have a pole at the cusp
$\infty$ of $\mathcal{F}$, and assume further that the Fourier expansion of $f$ at $\infty$ has its constant
term equal to one.  Let $P(w)$ be the Kronecker limit function associated to the parabolic Eisenstein series associated to $\mathrm{PSL}(2,\Z)$;
below we will write $P(w)$ in terms of the Dedekind Delta function, but for now we want to keep the concept
of a Kronecker limit function in the conversation.  With all this, the Rohrlich-Jensen formula
is the statement that
\begin{equation} \label{rohrl thm}
\frac{1}{2\pi}\int\limits_{\mathrm{PSL}(2,\Z)\backslash \h} \log|f(z)|d\mu(z) + \sum_{w\in \mathcal{F}}
\frac{\mathrm{ord}_w(f)}{\mathrm{ord}(w)}P(w) =0.
\end{equation}

In this expression, $\mathrm{ord}_w(f)$ denotes the order of $f$ at $w$ as a meromorphic function,
and $\mathrm{ord}(w)$ denotes the order of the action of $\mathrm{PSL}(2,\Z)$ on $\HH$.  As a means
by which one can see beyond the above setting, one can view \eqref{rohrl thm} as evaluating the inner product
$$
\langle 1,\log|f(z)| \rangle=\int\limits_{\mathrm{PSL}(2,\Z)\backslash \h} 1\cdot \log|f(z)|d\mu(z)
$$
within the Hilbert space of $L^{2}$ functions on $\mathrm{PSL}(2,\Z)\backslash \HH$.

There are various directions in which \eqref{rohrl thm} has been extended.  In \cite{Ro84}, Rohrlich described
the analogue of \eqref{rohrl thm} for general Fuchsian groups of the first kind and for meromorphic modular forms $f$ of
non-zero weight; see page 19 of \cite{Ro84}.  In \cite{HIvPT19} the authors studied the
quotient of hyperbolic three space when acted upon by the discrete group $\mathrm{PSL}(2,\mathcal{O}_K)$
where $\mathcal{O}_K$ denotes the ring of integers of an imaginary quadratic field $K$.  In that setting, the function
$\log \vert f \vert$ is replaced by a function which is harmonic at all but a finite number of points and
at those points the function has prescribed singularities.  As in \cite{Ro84}, the analogue of \eqref{rohrl thm} involves
a function $P$ which is constructed from a type of Kronecker limit formula.

In \cite{BK19} the authors returned to the setting of $\mathrm{PSL}(2,\Z)$ acting on $\HH$.
Let $q_{z}=e^{2\pi i z}$ be the standard local coordinate near $\infty$ of $\mathrm{PSL}(2,\Z)\backslash \HH$.  The
Hauptmodul $j(z)$ is the unique $\mathrm{PSL}(2,\Z)$ invariant holomorphic function on $\HH$ whose expansion near
$\infty$ is $j(z) = q_{z}^{-1} + o(q_{z}^{-1})$  as $z$ approaches $\infty$. Let $T_{n}$ denote the $n$-th Hecke
operator and set $j_{n}(z) = j|T_n (z)$.  The main results of \cite{BK19} are the derivation of formulas for
the regularized scalar product $\langle j_n(z),\log ((\Im (z))^k|f(z)|) \rangle$ where $f$ is a weight $2k$ meromorphic modular
form with respect to $\mathrm{PSL}(2,\Z)$.  Below we will discuss further the formulas from \cite{BK19} and describe
the way in which their results are natural extensions of \eqref{rohrl thm}.

\subsection{Revisiting Rohrlich's theorem}

The purpose of this article is to extend the point of view that the Rohlrich-Jensen formula is the
evaluation of a particular type of inner product.  To do so, we shall revisit the role of each of the two terms
$j|T_n (z)$ and $\log ((\Im (z))^k|f(z)|)$.

The function $j|T_n (z)$ can be characterized as the unique holomorphic function which is $\mathrm{PSL}(2,\Z)$
invariant on $\HH$ and whose expansion near $\infty$ is $q_{z}^{-n} + o(q_{z}^{-1})$.  These properties hold
for the special value $s=1$ of the Niebur-Poincar\'e series $F_{-n}^{\Gamma}(z,s)$, which is
defined in \cite{Ni73} for any Fuchsian group $\Gamma$ of the first kind with one cusp and discussed in section \ref{sect_NP-series} below.  As proved in \cite{Ni73}, for any $m\in\N$,
the Niebur-Poincar\'e series $F_{m}^{\Gamma}(z,s)$ is an eigenfunction of the hyperbolic Laplacian $\Delta_{\hyp}$;
specifically, we have that
$$
\Delta_{\hyp} F_{m}^{\Gamma}(z,s) = s(s-1)F_{m}^{\Gamma}(z,s).
$$
Also, $F_{m}^{\Gamma}(z,s)$ is orthogonal to constant functions.

Furthermore, if $\Gamma = \mathrm{PSL}(2,\Z)$,
then for any positive integer $n$ there is an explicitly computable constant $c_{n}$ such that
\begin{equation}\label{j_via_F}
F_{-n}^{\mathrm{PSL}(2,\Z)}(z,1) = \frac{1}{2\pi\sqrt{n}}j_{n}(z) + c_{n}.
\end{equation}
As a result, the Rohrlich-Jensen formula proved in \cite{BK19},
when combined with Rohrlich's formula from \cite{Ro84}, reduces to computing the regularized inner product of
$F_{-n}^{\mathrm{PSL}(2,\Z)}(z,1)$ with $\log ((\Im (z))^k|f(z)|)$.

As for the term $\log ((\Im (z))^k|f(z)|)$, we begin by recalling Proposition 12 from \cite{JvPS19}.  Let $2k\geq 4$ be
any even positive integer, and let $f$ be a weight $2k$ meromorphic form $f$ which is $\Gamma$ invariant and with $q$-expansion at $\infty$
that is normalized so its constant term is equal to one.  Set $\Vert f\Vert(z) = y^{k}\vert f(z) \vert$, where $z=x+iy$.  Let
$\mathcal{E}^{\mathrm{ell}}_{\Gamma,w}(z,s)$ be the elliptic Eisenstein
series associated to the aforementioned data; a summary of the relevant properties of $\mathcal{E}^{\mathrm{ell}}_{\Gamma,w}(z,s)$
is given in section \ref{ell_Eisen_series} below.  Then, in \cite{JvPS19} it is proved that one has  the asymptotic relation
\begin{equation} \label{ell kroneck limit one cusp}
\sum_{w\in \mathcal{F}_\Gamma} \mathrm{ord}_w(f) \mathcal{E}^{\mathrm{ell}}_{\Gamma,w}(z,s)=
-s\log\left( |f(z)| |\eta_{\Gamma,\infty}^4(z)|^{-k}\right) + O(s^2)
\,\,\,\,\,
\textrm{\rm as $s \to 0$}
\end{equation}
where $\mathcal{F}_\Gamma$ is the fundamental domain for the action of $\Gamma$ on $\mathbb{H}$ and $\eta_{\Gamma,\infty}(z)$ is the analogue of the classical eta function for the modular group, see the Kronecker limit formula \eqref{KronLimitPArGen} for the parabolic Eisenstein series.
With this, formula \eqref{ell kroneck limit one cusp} can be written as
\begin{equation} \label{ell kroneck limit one cusp2}
\log\left( \Vert f\Vert (z) \right) = kP_{\Gamma}(z) -  \sum_{w\in \mathcal{F}_\Gamma}
\mathrm{ord}_w(f) \lim_{s\to 0} \frac{1}{s} \mathcal{E}^{\mathrm{ell}}_{\Gamma,w}(z,s),
\end{equation}
where $P_{\Gamma}(z)=\log(|\eta_{\Gamma,\infty}^4(z)|\Im (z))$ is the Kronecker limit function associated to the
parabolic Eisenstein series $\E^{\mathrm{par}}_{\Gamma,\infty}(z,s)$; the precise normalizations and
expressions defining $\E^{\mathrm{par}}_{\Gamma,\infty}(z,s)$ will be clarified below.

Following \cite{CJS20}, one can recast \eqref{ell kroneck limit one cusp2} in terms of
the resolvent kernel, which we now shall undertake.

The resolvent kernel, also called the automorphic Green's function, $G_s^{\Gamma}(z,w)$ is the integral kernel
which for almost all $s \in \CC$ inverts the operator $\Delta_{\hyp} + s(s-1)$.  In other words,
$$
\Delta_{\hyp} G_s^{\Gamma}(z,w) = s(1-s) G_s^{\Gamma}(z,w).
$$
The resolvent kernel is closely related to the elliptic Eisenstein series; see \cite{vP16} as well as \cite{CJS20}.
Specifically, from Corollary 7.4 of \cite{vP16}, after taking into account a sign difference in our normalization,
we have that
\begin{equation}\label{Ell Green connection}
\mathrm{ord}(w) \mathcal{E}^{\mathrm{ell}}_{\Gamma,w}(z,s) = -\frac{2^{s+1}\sqrt{\pi} \Gamma(s+1/2)}{\Gamma(s)}G_s^{\Gamma}(z,w) +O(s^2)
\,\,\,\,\,
\textrm{\rm as $s \to 0$}
\end{equation}
for all $z,w \in\h$ with $z\neq \gamma w$ when $\gamma\in\Gamma$.
It is now evident that one can
express $\log\left( \Vert f \Vert(z) \right)$ as a type of Kronecker limit function.
Indeed, upon using the functional equation for the Green's function, we will prove below the following
result.  Under certain general conditions the form $f$, as described above, can be realized through a type of factorization
theorem, namely that
\begin{align}\label{log norm basic}
\log\left( \Vert f\Vert (z) \right)&=-2k + 2\pi \sum_{w\in \mathcal{F}_{\Gamma}} \frac{\mathrm{ord}_w(f)}{\mathrm{ord}(w)}
\lim_{s\to 1}\left(G_s^{\Gamma}(z,w) + \E_{\Gamma,\infty}^{\mathrm{par}}(z,s)\right)\notag \\ &= 2\pi \sum_{w\in \mathcal{F}_\Gamma}
\frac{\mathrm{ord}_w(f)}{\mathrm{ord}(w)} \left[ \lim_{s\to 1}\left(G_s^{\Gamma}(z,w) + \E_{\Gamma,\infty}^{\mathrm{par}}(z,s)\right) -
\frac{2}{\mathrm{vol}_{\hyp}(\Gamma \backslash \HH)} \right].
\end{align}

With all this, it is evident that one can view the inner product realization of the Rohrlich-Jensen
formula as a special value of the inner product of the Niebur-Poincar\'e series $F_{m}^\Gamma(z,s)$ and the resolvent
kernel $G_{s}^{\Gamma}(z,w)$ plus the parabolic Eisenstein series $\E_{\Gamma,\infty}^{\mathrm{par}}(z,s)$.  Furthermore, because all
terms are eigenfunctions of the Laplacian, one can seek to compute the inner product
in hand in a manner similar to that which yields the Maass-Selberg formula.

\subsection{Our main results}

Unless otherwise explicitly stated, we will assume for the remainder
of this article that $\Gamma$ is any Fuchsian group of the first kind with one cusp. By conjugating $\Gamma$, if necessary, we may assume that the cusp is at $\infty$, with the cuspidal width equal to one. The group $\Gamma$ will be arbitrary, but fixed, throughout this article, so, for the sake of brevity, in the sequel, we will suppress the index $\Gamma$ in the notation for Eisenstein series, the Niebur-Poincar\'e series, the Kronecker limit function, the fundamental domain and the resolvent kernel.
When $\Gamma$ is taken to be the modular group or the Atkin-Lehner group, that will be indicated in the notation.

With the above discussion, we have established that one manner in which the Rohrlich-Jensen formula can be
understood is through the study of the regularized inner product
\begin{equation}\label{RJ_integral}
\langle F_{-n}(\cdot,1), \overline{\lim_{s\to 1} \left(G_s(\cdot,w) + \E_\infty^{\mathrm{par}}(\cdot,s)\right)} \rangle,
\end{equation}
which is defined as follows.  Since $\Gamma$ has one cusp at $\infty$, one can construct a (Ford) fundamental domain $\mathcal{F}$
of the action of $\Gamma$ on $\HH$.  Let $M = \Gamma \backslash \HH$.  A cuspidal neighborhood $\mathcal{F}_\infty(Y)$ of $\infty$
is given by $0 < x \leq 1$ and $y \geq Y$, where $z=x+iy$ and some $Y \in \RR$ sufficiently large.
(We recall that we have normalized the cusp to be of width one.)  Let
$\mathcal{F}(Y) = \mathcal{F}\setminus \mathcal{F}_\infty(Y)$.  Then, we define \eqref{RJ_integral} to be
$$
\lim_{Y\to \infty} \int\limits_{\mathcal{F}(Y)}F_{-n}(z,1)\lim_{s\to 1} \left(G_s(z,w) + \E_\infty^{\mathrm{par}}(z,s)\right)d\mu_{\hyp}(z)
$$
where $d\mu_{\hyp}(z)$ denotes the hyperbolic volume element.  The function
$G_s(z,w) +  \E_\infty^{\mathrm{par}}(z,s)$ is unbounded as $z \to w$.  However, the asymptotic growth of the function is logarithmic
thus integrable, hence it is not necessary to regularize the integral in \eqref{RJ_integral} in a neighborhood containing $w$.  The need to
regularize the inner product \eqref{RJ_integral} stems solely the from the exponential growth behavior of the factor $F_{-n}(z,1)$ as $z\to\infty$.

Our first main result of this article is the following theorem.

\begin{theorem}\label{thm:main}
For any positive integer $n$ and any point $w\in \mathcal{F}$
\begin{equation}\label{main f-la}
\langle F_{-n}(\cdot,1), \overline{\lim_{s\to 1} \left(G_s(\cdot,w) + \E_\infty^{\mathrm{par}}(\cdot,s)\right)} \rangle
= - \frac{\partial}{\partial s} F_{-n}(w,s) \Big|_{s=1}.
  \end{equation}
\end{theorem}

We can combine Theorem \ref{thm:main} with the factorization theorem \eqref{log norm basic} and
properties of $F_{-n}(z,1)$ proved in \cite{Ni73} and obtain the following extension of the Rohrlich-Jensen formula.

\begin{corollary}\label{Rohr-Jensen}
In addition to the notation above, assume that the even weight $2k\geq 0$ meromorphic form $f$ has been normalized so its $q$-expansion
at $\infty$ has constant term equal to $1$.  Then we have that

\begin{equation}\label{main f-la - corollary}
\langle F_{-n}(\cdot ,1), \log\Vert f\Vert  \rangle = - 2\pi \sum_{w\in \mathcal{F}} \frac{\mathrm{ord}_w(f)}{\mathrm{ord}(w)}\frac{\partial}{\partial s}\left. F_{-n}(w,s) \right|_{s=1}.
  \end{equation}
\end{corollary}

Let $g$ be a $\Gamma$ invariant analytic function which necessarily has a pole at $\infty$.  As such, there is a positive
integer $K$ and set of complex numbers $\{a_{n}\}_{n=1}^{K}$ such that
$$
g(z) = \sum_{n=1}^K a_n q_z^{-n}  + O(1)
\,\,\,\,\,
\textrm{as $z \rightarrow \infty$.}
$$
It is proved in \cite{Ni73} that
\begin{equation}\label{g expr}
g(z) = \sum_{n=1}^K 2\pi \sqrt{n}a_n F_{-n}(z,1) + c(g)
\end{equation}
for some constant depending only upon $g$.  With this, we can combine Corollary \ref{Rohr-Jensen}
and the Theorem on page 19 of \cite{Ro84} to obtain the following result.

\begin{corollary} \label{cor:j-function}
With notation as above, there is a constant $\beta$, defined by the Laurent expansion of $\E_\infty^{\mathrm{par}}(z,s)$
near $s=1$, such that
\begin{multline}\label{main f-la - corollary 2}
\langle g, \log\Vert f\Vert \rangle = - 2\pi \sum_{w\in \mathcal{F}} \frac{\mathrm{ord}_w(f)}{\mathrm{ord}(w)}
\Bigg(2\pi \sum_{n=1}^K \sqrt{n}a_n\frac{\partial}{\partial s} F_{-n}(w,s) \Big|_{s=1}  \\   + c(g)(P(w) - \beta \vol_{\hyp}(M) +2) \Bigg).
\end{multline}
\end{corollary}

The constant $\beta$ is given in \eqref{KronLimitPArGen}.  We refer the reader to equation \eqref{KronLimitPArGen} for further
details regarding the normalizations which define $\beta$ and the parabolic Kronecker limit function $P$.

Finally, we will consider the generating function of the normalized series constructed from the right-hand side of \eqref{main f-la}.
Specifically, we will prove the following identity.

\begin{theorem} \label{thm:generating series}
With notation as above, the generating series
$$
\sum_{n\geq 1}2\pi \sqrt{n} \frac{\partial}{\partial s} F_{-n}(w,s) \Big|_{s=1}q_z^n
$$
is, in the $z$ variable, the holomorphic part of the weight two biharmonic Maass form
$$
\mathcal{G}_w(z):=i\frac{\partial}{\partial z} \left( \frac{\partial}{\partial s}
\left( G_s(z,w) + \E_\infty^{\mathrm{par}}(w,s)\right) \Big|_{s=1} \right).
$$
\end{theorem}
Note that the weight two biharmonic Maas form is a function which satisfies the weight two modularity in $z$ and which is annihilated by $\Delta_2^2=(\xi_0\circ \xi_2)^2$, where, classically $\xi_\kappa := 2iy^\kappa \overline{\frac{\partial}{\partial \overline{z}}}$. It is clear from the definition that $\mathcal{G}_w(z)$ satisfies the weight two modularity in the $z$ variable. In section \ref{sect proof of thm 4} we will prove that $(\xi_0\circ \xi_2)^2\mathcal{G}_w(z)=0$.

In the case $\Gamma = \text{\rm PSL}(2,\ZZ)$, our results will generalize the main theorems from
\cite{BK19}, as we will discuss below.

\subsection{Outline of the paper}

In section 2 we will establish notation and recall certain results from the literature.
There are two specific examples of Poincar\'e series which are particularly important for our
study, the Niebur-Poincar\'e series and the resolvent kernel.  Both series are defined, and basic properties
are presented in section 3.  In section 4 we state the Kronecker limit formulas associated to parabolic
and elliptic Eisenstein series, and we prove the factorization theorem \eqref{log norm basic}.  The
proofs of the main results listed above will be given in section 5.

To illustrate our results, various examples are given in section 6.  Our first example is when
$\Gamma = \text{\rm PSL}(2,\ZZ)$ where, as claimed above, our results yield the main theorems of \cite{BK19}.
We then turn to the case when $\Gamma$ is an Atkin-Lehner group $\Gamma_0(N)^+$ for square-free level $N$.
The first examples are when the genus of $\Gamma_0(N)^+$ is zero and when the function $g$
in Corollary \ref{cor:j-function} is the Hauptmodul $j_N^+(z)$.  The next two examples we present are
for levels $N=37$ and $N=103$.  For these levels the genus of the quotient by $\Gamma_0(N)^+$ is one and two, respectively.  In these cases,
certain generators of the corresponding function fields were constructed in \cite{JST13}.  Consequently, we are able to
employ the results from \cite{JST13} and fully develop Corollary \ref{cor:j-function}.

 \section{Background material}

\subsection{Basic notation} \label{notation}
Let $\Gamma\subset\text{\rm PSL}(2,\R)$ denote a Fuchsian
group of the first kind acting by fractional
linear transformations on the hyperbolic upper half-plane $\mathbb{H}:=\{z=x+iy\in\mathbb{C}\,
|\,x,y\in\mathbb{R};\,y>0\}$. We let $M:=\Gamma\backslash\mathbb{H}$, which is a finite
volume hyperbolic Riemann surface, and denote by $p:\mathbb{H}\longrightarrow M$
the natural projection. We assume that $M$ has $e_{\Gamma}$
elliptic fixed points and one cusp at $\infty$ of width one. By an abuse of notation, we also say that $\Gamma$ has a cusp at $\infty$ of width one, meaning that the stabilizer $\Gamma_\infty$ of $\infty$ is generated by the matrix $\bigl(\begin{smallmatrix}
1&1\\0&1\end{smallmatrix}\bigr)$. We identify $M$
locally with its universal cover $\mathbb{H}$.  By $\mathcal{F}$ we denote the ``usual'' (Ford) fundamental domain for $\Gamma$ acting
on $\mathbb H$.

We let $\mu_{\mathrm{hyp}}$ denote the hyperbolic metric on $M$, which is compatible with the
complex structure of $M$, and has constant negative curvature equal to minus one.
The hyperbolic line element $ds^{2}_{\hyp}$, resp.~the hyperbolic Laplacian
$\Delta_{\hyp}$ acting on functions, are given in the coordinate $z=x+iy$ on $\mathbb{H}$ by
\begin{align*}
ds^{2}_{\hyp}:=\frac{dx^{2}+dy^{2}}{y^{2}},\quad\textrm{resp.}
\quad\Delta_{\hyp}:=-y^{2}\left(\frac{\partial^{2}}{\partial
x^{2}}+\frac{\partial^{2}}{\partial y^{2}}\right).
\end{align*}
By $d_{\mathrm{hyp}}(z,w)$ we denote the hyperbolic distance between to the two points $z\in\mathbb{H}$ and
$w\in\mathbb{H}$. Our normalization of the hyperbolic Laplacian is different from the one considered in \cite{Ni73} and \cite{He83} where the Laplacian is taken with the plus sign.

\subsection{Modular forms}

Following \cite{Se73}, we define a weakly modular form $f$ of even weight $2k$ for $k \geq  0$ associated to $\Gamma$ to be a
function $f$ which is meromorphic on $\mathbb H$ and satisfies the transformation property
\begin{equation}\label{transf prop}
f\left(\frac{az+b}{cz+d}\right) = (cz+d)^{2k}f(z),
\quad\textrm{for any $\begin{pmatrix}a&b\\c&d\end{pmatrix} \in \Gamma$.}
\end{equation}

In the setting of this paper, any weakly modular
form $f$ will satisfy the relation $f(z+1)=f(z)$, so that for some positive integer $N$ we can write
$$
f(z) = \sum\limits_{n=-N}^{\infty}a_{n}q_z^{n},
\quad\text{ where } q_z =e(z)= e^{2\pi iz}.
$$
If $a_{n} = 0$ for all $n < 0$, then $f$ is said to be holomorphic at the cusp at $\infty$.
A holomorphic modular form with respect to $\Gamma$ is a weakly modular form which is holomorphic on $\mathbb H$ and at all the cusps of $\Gamma$.

When the weight $k$ is zero, the transformation property \eqref{transf prop} indicates that the function $f$ is invariant with respect to the action of elements of the group $\Gamma$, so it may be viewed as a meromorphic function on the surface $M=\Gamma\backslash \mathbb{H}$. In other words, a meromorphic function on $M$ 
is a weakly modular form of weight $0$.

For any two weight $2k$ weakly modular forms $f$ and $g$ associated to $\Gamma$, with integrable singularities at finitely many points in $\mathcal{F}$,
the generalized inner product $\langle \cdot,\, \cdot \rangle$ is defined as
\begin{equation} \label{def:inner prod}
\langle f,g\rangle = \lim_{Y\to \infty} \int\limits_{\mathcal{F}(Y)}f(z) \overline{g(z)}(\text{\rm Im}(z))^{2k}d\mu_{\hyp}(z)
\end{equation}
where the integration is taken over the portion  $\mathcal{F}(Y)$ of the fundamental domain $\mathcal{F}$ equal
to $\mathcal{F}\setminus \mathcal{F}_\infty(Y)$.


\subsection{Atkin-Lehner groups} \label{sect Atkin Leh groups}

Let $N=p_1\cdot\ldots\cdot p_r$ be a square-free, non-negative integer including the case $N=1$.
The subset of $\text{\rm SL}(2,\R)$, defined by
\begin{align*}
\Gamma_0(N)^+:=\left\{ \frac{1}{\sqrt{e}}\begin{pmatrix}a&b\\c&d\end{pmatrix}\in
  \text{\rm SL}(2,\R): \,\,\, ad-bc=e, \,\,\, a,b,c,d,e\in\Z, \,\,\, e\mid N,\ e\mid a,
   \ e\mid d,\ N\mid c \right\}
\end{align*}
is an arithmetic subgroup of $\text{\rm SL}(2,\R)$. We use the terminology Atkin-Lehner groups of level $N$ to describe $\Gamma_0(N)^+$
 in part because these groups are obtained by adding all Atkin-Lehner involutions to the congruence group $\Gamma_0(N)$, see \cite{AtLeh70}.
Let $\{\pm \textrm{Id}\}$ denote the set of two elements where $\textrm{Id}$ is the identity matrix.
In general, if $\Gamma$ is a subgroup of $\text{\rm SL}(2,\R)$, we let $\overline{\Gamma} := \Gamma /\{\pm \textrm{Id}\}$ denote its projection into $\textrm{PSL}(2,\R)$.

Set $Y_N^{+}:=\overline{\Gamma_0(N)^+} \backslash \h$. According to \cite{Cum04},
for any square-free $N$ the quotient space $Y_{N}^{+}$ has one cusp at $\infty$ with the cusp width equal to one. The spaces $Y_{N}^{+}$ will be
used in the last section where we give examples of our results for generators of function fields
of meromorphic functions on $Y_N^{+}$.

\subsection{Generators of function fields of Atkin-Lehner groups of small genus}

An explicit construction of generators of function fields of all meromorphic functions on $Y_N^{+}$ with genus $g_{N,+}\leq 3$ was given in \cite{JST13}.

When $g_{N,+}=0$, the function field of meromorphic functions on $Y_N^+$ is generated by a single function, the Hauptmodul $j_N^+(z)$, which is
normalized so that its $q$-expansion is of the form $q_z^{-1}+O(q_z)$. The Hauptmodul $j_N^+(z)$ appears in the ``Monstrous Moonshine'' and was investigated
in many papers, starting with Conway and Norton \cite{CN79}. The action of the $m$-th Hecke operator $T_m$ on $j_N^+(z)$ produces a meromorphic form on
$Y_N^{+}$ with the $q$-expansion $j_N^+|T_{m}(z)= q_z^{-m} + O(q_z)$.

When $g_{N,+}\geq 1$, the function field associated to $Y_N^+$ is generated by two functions $x_N^+(z)$ and $y_N^+(z)$. Stemming
from the results in \cite{JST13}, we have that for $g_{N,+}\leq 3$
the generators $x_N^+(z)$ and $y_N^+(z)$ such that their $q$-expansions are of the form
$$
x_N^+(z)=q_z^{-a}+\sum_{j=1}^{a-1}a_jq_z^{-j}+O(q_z) \quad \text{and} \quad y_N^+(z)=q_z^{-b}+\sum_{j=1}^{b-1}b_jq_z^{-j}+O(q_z)
$$
where $a,b$ are positive integers with $a\leq 1+g_{N,+}$, and $b\leq 2+g_{N,+}$.  Furthermore, for $g_{N,+} \leq 3$,  it is shown in \cite{JST13} that
all coefficients in the $q$-expansion for $x_N^+(z)$ and $y_N^+(z)$ are integers.  For all such $N$, the precise values of
these coefficients out to large order were computed, and the results are available at \cite{jst url}.

\section{Two Poincar\'e series}

In this section we will define the Niebur-Poincar\'e series $F_{m}(z,s)$ and the resolvent kernel, also referred
to as the automorphic Green's function $G_s(z,w)$.  We refer the reader to \cite{Ni73} for additional information
regarding $F_{m}(z,s)$ and to \cite{He83} and \cite{Iwa02} and references therein for further details regarding $G_s(z,w)$. As said above, we will suppress the group $\Gamma$ from the notation.

\subsection{Niebur-Poincar\'e series}\label{sect_NP-series}

We start with the definition and properties of the Niebur-Poincar\'e series $F_{m}(z,s)$ associated to a co-finite Fuchsian group with one cusp; then we will specialize results to the setting of Atkin-Lehner groups.

\subsubsection{Niebur-Poincar\'e series associated to a co-finite Fuchsian group with one cusp}

Let $m$ be a non-zero integer, $z=x+iy\in\h$, and $s\in\C$ with $\Re(s)>1$.
Recall the notation $e(x):=\exp(2\pi i x)$, and let $I_{s-1/2}$ denote the modified $I$-Bessel function of the first kind; see, for example Appendix B.4,
formula (B.32) of \cite{Iwa02}). The Niebur-Poincar\'e series $F_{m}(z,s)$ is defined by the series
\begin{equation}\label{Def:Niebur Poinc series}
  F_m(z,s)=F_m^\Gamma(z,s):= \sum_{\gamma\in\Gamma_\infty \backslash \Gamma} e(m\Re(\gamma z))(\Im(\gamma z))^{1/2}I_{s-1/2}(2\pi |m| \Im(\gamma z)).
\end{equation}
For fixed $m$ and $z$, the series \eqref{Def:Niebur Poinc series} converges absolutely and uniformly on any compact subset of the half plane
$\Re(s)>1$. Moreover, $\Delta_{\hyp}F_m(z,s) = s(1-s) F_m(z,s)$ for all $s\in\C$ in the half plane $\Re(s)>1$. From Theorem 5 of \cite{Ni73},
we have that for any non-zero integer $m$, the function $F_m(z,s)$ admits a meromorphic continuation to the whole complex plane $s\in\CC$.
Moreover, $F_m(z,s)$ is holomorphic at $s=1$ and, according to the spectral expansion given in Theorem 5 of \cite{Ni73}, $F_m(z,1)$
is orthogonal to constant functions, meaning that
$$
\langle F_m(z,1), 1 \rangle=0.
$$

For our purposes, it is necessary to employ the Fourier expansion of $F_{m}(z,s)$ in the cusp $\infty$.  The
Fourier expansion is proved in \cite{Ni73} and involves Kloosterman sums $S(m,n;c)$, which we now define.  For
any integers $m$ and $n$, and real number $c$, define
$$
S(m,n;c)=S_\Gamma(m,n;c):= \sum_{\bigl(\begin{smallmatrix}
a&\ast\\c&d\end{smallmatrix}\bigr)\in \Gamma_\infty  \diagdown \Gamma  \diagup \Gamma_\infty } e\left( \frac{ma + nd}{c}\right).
$$
For  $\Re(s)>1$ and $z=x+iy\in \h$, the Fourier expansion of $F_m(z,s)$ is given by
\begin{equation}\label{Four exp Nieb}
F_m(z,s)=e(mx)y^{1/2}I_{s-1/2}(2\pi |m|y) + \sum_{k=-\infty}^{\infty}b_k(y,s;m)e(kx),
\end{equation}
where
$$
b_0(y,s;m) = \frac{y^{1-s}}{(2s-1)\Gamma(s)}2\pi^s |m|^{s-1/2} \sum_{c>0} S(m,0;c)c^{-2s}=\frac{y^{1-s}}{(2s-1)}B_0(s;m)
$$
and, for $k\neq 0$
$$
b_k(y,s;m)=B_k(s;m)y^{1/2}K_{s-1/2}(2\pi |m|y),
$$
with
$$
B_k(s;m)= 2 \sum_{c>0}S(m,k;c)c^{-1}\cdot \left\{
                                            \begin{array}{ll}
                                              J_{2s-1}\left(\frac{4\pi}{c} \sqrt{mk}\right), & \textrm{\rm if \,}mk>0 \\
                                              I_{2s-1}\left(\frac{4\pi}{c} \sqrt{|mk|}\right), & \textrm{\rm if \,} mk<0.
                                            \end{array}
                                          \right.
$$
In the above expression, $J_{2s-1}$ denotes the $J$-Bessel function and $K_{s-1/2}$ is the modified Bessel function;
see, for example, formula (B.28) in \cite{Iwa02} for $J_{2s-1}$  and formula (B.34) of \cite{Iwa02})
for $K_{s-1/2}$.

According to the proof of Theorem 6 from \cite{Ni73}, the Fourier expansion \eqref{Four exp Nieb} extends by the principle of analytic continuation to the case when $s=1$, hence putting $B_k(1;m):= \lim_{s\downarrow 1} B_k(s;m)$, we have
\begin{equation}\label{Four exp Nieb at 1}
F_m(z,1)=\frac{\sinh(2\pi|m|y)}{\pi \sqrt{|m|}}e(mx) + B_0(1;m)+ \sum_{k\in\Z\setminus\{0\}}\frac{1}{2 \sqrt{|k|}} e^{-2\pi|k|y}B_k(1;m)e(kx).
\end{equation}
It is clear from \eqref{Four exp Nieb at 1} that for $n>0$ one has that
$$
F_{-n}(z,1) = \frac{1}{2\pi\sqrt{n}}q_{z}^{-n} + O(1)
\,\,\,\,\,\textrm{\rm as $z \rightarrow \infty$.}
$$
Moreover, applying $\frac{\partial }{\partial s}$ to the Fourier expansion \eqref{Four exp Nieb}, taking $s=1$ and reasoning analogously as in the proof of Lemma 4.3. (1), p. 19  of \cite{BK19}
we immediately deduce the following crude bound
\begin{equation} \label{eq: N-P deriv bound}
\left.\frac{\partial}{\partial s} F_{-n}(z,s)\right|_{s=1} \ll \exp\left( 2\pi n \Im(z)\right), \quad \text{as} \quad \Im(z) \to \infty.
\end{equation}

We note that the value of the derivative of the Niebur-Poincar\'e series at $s=1$ satisfies a
differential equation, namely that
\begin{align} \notag
 \Delta_{\hyp}\left(\frac{\partial}{\partial s}\left. F_{-n}(z,s) \right|_{s=1} \right) &=\lim_{s\to 1}
 \Delta_{\hyp} \left(\frac{F_{-n}(z,s) - F_{-n}(z,1)}{(s-1)}\right)= \\&=\lim_{s\to 1} \left(\frac{s(1-s)F_{-n}(z,s) -0}{(s-1)}\right) = -F_{-n}(z,1). \label{delta of deriv of N-P}
\end{align}

\subsubsection{Fourier expansion when $\Gamma$ is an Atkin-Lehner group}

One can explicitly evaluate $B_0(1;m)$ for $m > 0$ when $\Gamma$ is an Atkin-Lehner group. Set  $\Gamma=\overline{\Gamma_0(N)^+}$
where $N$ is a squarefree, which we express as $N=\prod\limits_{\nu=1}^r p_\nu$.  Let $B_{0,N}^+(1;m)$ denote
the coefficient $B_0(1;m)$ for $\overline{\Gamma_0(N)^+}$.

From Theorem 8 and Proposition 9 of \cite{JST13} we get that
\begin{equation}\label{B0 for A-L groups}
B_{0,N}^+(1;m)= \frac{12\sigma(m)}{\pi\sqrt{m}}\prod\limits_{\nu=1}^r \left(1-
\frac{p_\nu^{\alpha_{p_\nu}(m)+1}(p_\nu -1) }{\left(p_\nu^{\alpha_{p_\nu}(m)+1} - 1\right)(p_\nu+1)}\right) ,
\end{equation}
where $\sigma(m)$ denotes the sum of divisors of a positive integer $m$ and $\alpha_p(m)$ is the largest integer such that $p^{\alpha_p(m)}$ divides $m$.
These expressions will be used in our explicit examples in section \ref{sect:examples} below.

\subsection{Automorphic Green's function} \label{sec: aut Green}

The automorphic Green's function, also called the resolvent kernel, for the Laplacian on $M$ is defined on page 31 of  \cite{He83}.
In the notation of \cite{He83}, let $\chi$ be the identity character, $z,w\in \mathcal{F}$ with $z\neq w$, and $s\in\C$ with $\Re(s)>1$.
Formally, consider the series
$$
G_s(z,w)= \sum_{\gamma\in\Gamma} k_s(\gamma z,w)
$$
with
$$
k_s(z,w):=-\frac{\Gamma(s)^2}{4\pi \Gamma(2s)}\left[1-\left| \frac{z-w}{z-\overline{w}}\right|^2\right]^s
F\left(s,s;2s;1-\left| \frac{z-w}{z-\overline{w}}\right|^2\right)
$$
and where $F(\alpha,\beta;\gamma;u)$ is the classical hypergeometric function. We should point out that
the normalization we are using, which follows \cite{He83}, differs from the normalization for
the Green's function in Chapter 5 of \cite{Iwa02}; the two normalizations differ by a minus sign.
With this said, it is proved in \cite{He83}, Proposition 6.5. on p.33 that the series which defines $G_{s}(z,w)$ converges
uniformly and absolutely on compact subsets of $(z,w,s) \in \mathcal{F} \times \mathcal{F}\times \{s\in\C :\Re(s)>1\}$.

Furthermore, for all $s\in\C$ with $\Re(s)>1$, and all $z,w \in\h$ with $z\neq \gamma w$ for $\gamma\in\Gamma$, the function
$G_s(z,w)$ is the eigenfunction of $\Delta_{\hyp}$ associated to the eigenvalue $s(1-s)$.

Combining formulas 9.134.1. and 8.703. from \cite{GR07} and applying the identity
$$
\cosh(d_{\hyp}(z,w))=\left(2-\left[1-\left| \frac{z-w}{z-\overline{w}}\right|^2\right]\right)\left(1-\left| \frac{z-w}{z-\overline{w}}\right|^2\right)^{-1}
$$
we deduce that $$k_s(z,w)=-\frac{1}{2\pi}Q^0_{s-1}(\cosh(d_{\hyp}(z,w))),$$ where $Q_\nu^{\mu}$ is the associated Legendre function as defined by formula 8.703 in \cite{GR07}, with $\nu=s-1$ and $\mu=0$.

Now, we can combine Theorem 4 of \cite{Ni73}
with Theorem 5.3 of \cite{Iwa02}, to deduce the Fourier expansion of the automorphic Green function in terms of
the Niebur-Poincar\'e series.
Specifically, let $w\in \mathcal{F}$ be fixed. Assume $z \in \mathcal{F}$
with $y=\Im(z) > \max\{\Im(\gamma w): \gamma\in \Gamma\}$, and assume $s\in\C$ with $\Re(s)> 1$.  Then
$G_{s}(z,w)$ admits the expansion
\begin{equation}\label{Four exp Green}
G_s(z,w)=-\frac{y^{1-s}}{2s-1}\E_\infty^{\mathrm{par}}(w,s)-\sum_{k\in\Z\smallsetminus \{0\}} y^{1/2}K_{s-1/2}(2\pi |k| y) F_{-k}(w,s)e(kx)
\end{equation}
where $\E_\infty^{\mathrm{par}}(w,s)$ is the parabolic Eisenstein series associated to the cusp at $\infty$ of $\Gamma$, see the next section for its full description.

Function $G_s(z,w)$ is unbounded as $z\to w$ and, according to Proposition 6.5. from \cite{He83} we have the asymptotics
$$
G_s(z,w)=\frac{\mathrm{ord}(w)}{2\pi}\log|z-w|+O(1),\quad \text{as}\quad z\to w.
$$

\section{Eisenstein series and their Kronecker limit formulas}

The purpose of this section is two-fold.  First, we state the definitions of parabolic and elliptic Eisenstein
series as well as their associated Kronecker limit formulas.  Specific examples of the parabolic Kronecker limit
formulas are recalled from \cite{JST13}.  Second, we prove the factorization theorem for meromorphic forms in
terms of elliptic Kronecker limit functions, as stated in \eqref{ell kroneck limit one cusp2}.

\subsection{Parabolic Kronecker limit functions}

Associated to the cusp at $\infty$ of $\Gamma$ one has a parabolic Eisenstein series ${\cal E}^{\mathrm{par}}_{\infty}(z,s)$.
Let $\Gamma_{\infty}$ denote the stabilizer subgroup within $\Gamma$ of $\infty$.  For $z\in \h$ and $s \in \mathbb{C}$ with $\textrm{Re}(s) > 1$,
${\cal E}^{\mathrm{par}}_{\infty}(z,s)$ is defined by the series
\begin{equation*}
{\cal E}^{\mathrm{par}}_{\infty}(z,s) =
\sum\limits_{\gamma \in \Gamma_{\infty}\backslash \Gamma}\textrm{Im}(\gamma z)^{s}.
\end{equation*}
It is well-known that ${\cal E}^{\mathrm{par}}_{\infty}(z,s)$ admits a meromorphic continuation
to all $s\in \CC$ and a functional equation in $s$.

For us, the Kronecker limit formula means the determination of the constant term in the Laurent expansion
of ${\cal E}^{\mathrm{par}}_{\infty}(z,s)$ at $s=1$.
Classically, Kronecker's limit formula  is the assertion that for $\Gamma = \textrm{PSL}(2,\mathbb{Z})$
one has  that
\begin{equation}\label{PSL2_KLF}
\mathcal{E}^{\mathrm{par}}_{\infty}(z,s)=
\frac{3}{\pi(s-1)}
-\frac{1}{2\pi}\log\bigl(|\Delta(z)|\Im(z)^{6}\bigr)+C+O(s-1)
\,\,\,\text{\rm as} \,\,\,
s \rightarrow 1.
\end{equation}
where $C=6(1-12\,\zeta'(-1)-\log(4\pi))/\pi$ and $\Delta(z)$ is Dedekind's Delta function which defined by
\begin{equation}\label{PSL2_Delta}
\Delta(z) = \left[q_{z}^{1/24}\prod\limits_{n=1}^{\infty}\left(1 - q_{z}^{n}\right)\right]^{24} = \eta(z)^{24}.
\end{equation}
We refer to \cite{Siegel80} for a proof of \eqref{PSL2_KLF}, though the above formulation
follows the normalization from \cite{JST13}.

For general Fuchsian groups of the first kind, Goldstein \cite{Go73} studied analogues of the Kronecker's limit formula associated to parabolic Eisenstein series.  After a slight renormalization and trivial generalization, Theorem 3-1 from \cite{Go73} asserts that the parabolic
Eisenstein series $\E^{\mathrm{par}}_{\infty}(z,s)$ admits the Laurent expansion
\begin{equation} \label{KronLimitPArGen}
\E^{\mathrm{par}}_{\infty}(z,s)= \frac{1}{\vol_{\hyp}(M) (s-1)} + \beta- \frac{1}{\vol_{\hyp}(M)} \log (\vert\eta_{\infty}^4(z)\vert \Im(z)) + O(s-1),
\end{equation}
as  $s \to 1$ and where $\beta=\beta_{\Gamma}$ is a certain real constant depending only on the group $\Gamma$.
As the notation suggests, the function $\eta_{\infty}(z)$ is a holomorphic form for $\Gamma$ and can be viewed as
a generalization of the eta function $\eta(z)$ which is defined in \eqref{PSL2_Delta} for the full modular group.

By employing the functional equation for the parabolic Eisenstein series, as stated in Theorem 6.5 of  \cite{Iwa02},
one can re-write the Kronecker limit formula as stating that
\begin{equation} \label{KronLimas s to 0}
\E^{\mathrm{par}}_{\infty}(z,s)= 1+ \log (\vert \eta_{\infty}^4(z)\vert \Im(z))\cdot s +  O(s^2) \quad\text{  as  } s \to 0,
\end{equation}
see Corollary 3 of \cite{JvPS19}. In this formulation, we will call the function
$$
P(z)=P_{\Gamma}(z):=\log (\vert \eta_{\infty}^4(z)\vert \Im(z))
$$
the parabolic Kronecker limit function of $\Gamma$.

\subsection{Atkin-Lehner groups} \label{sect Atkin Leh groups_KLF}

Let $N=p_1\cdot \ldots \cdot p_r$ be a positive squarefree number, which includes the possibility that $N=1$ and set
$$
\ell_N = 2^{1-r}\textrm{lcm}\Big(4,\ 2^{r-1}\frac{24}{(24,\sigma(N))}\Big)
$$
where $\textrm{lcm}$ stands for the least common multiple of two numbers. In \cite{JST13}, Theorem 16, it is proved that
\begin{equation}\label{DeltaN}
\Delta_N(z):=\left( \prod_{v \mid N} \eta(v z) \right)^{\ell_N}
\end{equation}
is a weight $k_N=2^{r-1} \ell_N$ holomorphic form for $\Gamma_0(N)^+$ vanishing only at the cusp. By the valence formula,
the order of vanishing of $\Delta_N(z)$ at the cusp is $\nu_N:=k_N \vol_{\hyp}(Y_N^{+})/(4\pi)$
where $\vol_{\hyp}(Y_N^{+})=\pi\sigma(N)/(3\cdot 2^r)$ is the hyperbolic volume of the surface $Y_N^{+}$.

The Kronecker limit formula \eqref{KronLimitPArGen} for the parabolic Eisenstein series $\E^{\mathrm{par},N}_{\infty}(z,s)$
associated to $Y_N^{+}$ reads as
\begin{equation} \label{KronLimitPArGen - level N}
\E^{\mathrm{par},N}_{\infty}(z,s)= \frac{1}{\vol_{\hyp}(Y_N^{+}) (s-1)} + \beta_N - \frac{1}{\vol_{\hyp}(Y_N^{+})}P_N(z) + O((s-1))
\end{equation}
as  $s \to 1$.  From Example 7 and Example 4 of \cite{JvPS19} we have the explicit evaluations of $\beta_{N}$ and $P_{N}(z)$.
Namely,
\begin{equation} \label{betaN}
\beta_N=- \frac{1}{\vol_{\hyp} (Y_N^{+}) }\left( \sum_{j=1}^{r} \frac{(p_j -1)\log p_j}{2(p_j+1)}- \log N + 2\log (4\pi) + 24\zeta'(-1) - 2\right)
\end{equation}
and the parabolic Kronecker limit function $P_N(z)$ is given by
$$
P_N(z)= \log\left( \sqrt[2^r]{\prod_{v \mid N}  \vert \eta(vz)\vert ^4} \cdot \Im(z) \right).
$$

\subsection{Elliptic Kronecker limit functions}\label{ell_Eisen_series}

Elliptic subgroups of $\Gamma$ have finite order and a unique fixed point within $\mathbb H$. For all
but a finite number of $w \in \mathcal{F}$, the order of the elliptic subgroup $\Gamma_{w}$ which fixes $w$ is one.
For $z\in \h$ with $z\not=w$ and $s \in \mathbb{C}$ with $\textrm{Re}(s) > 1$, the elliptic Eisenstein series
${\cal E}^{\textrm{ell}}_{w}(z,s)$ is defined by the series
\begin{equation}\label{ell_eisen}
{\cal E}^{\textrm{ell}}_{w}(z,s) =\sum\limits_{\gamma \in \Gamma_{w}\backslash \Gamma}
\sinh(d_{\mathrm{hyp}}(\gamma z, w))^{-s} =
\sum\limits_{\gamma \in \Gamma_{w}\backslash \Gamma}
\left( \frac{2\,\textrm{Im}(w)\textrm{Im}(\gamma z)}{|\gamma z-w|\,|\gamma z-\overline{w}|} \right)^s.
\end{equation}
It was first shown in \cite{vP10} that \eqref{ell_eisen} admits a meromorphic continuation to all $s \in \CC$.

The analogue of the Kronecker limit formula for ${\cal E}^{\textrm{ell}}_{w}(z,s)$ was first proved in \cite{vP10}; see also \cite{JvPS19}.
In the setting of this paper, it is shown in \cite{vP10}
that for any $w\in\mathcal{F}$ the series \eqref{ell_eisen} admits the Laurent expansion
\begin{multline}\label{Kronecker_elliptic}
\mathrm{ord}(w)\,\mathcal{E}^{\mathrm{ell}}_{w}(z,s)-
\frac{2^{s}\sqrt{\pi}\,\Gamma(s-\frac{1}{2})}{\Gamma(s)}\mathcal{E}^{\mathrm{par}}_{\infty}(w,1-s)
\,\mathcal{E}^{\mathrm{par}}_{\infty}(z,s)= \\
=-\frac{2\pi}{\vol_{\hyp}(M)}
-\frac{2\pi}{\vol_{\hyp}(M)}\log\bigl(|H_{\Gamma}(z,w)|^{\mathrm{ord}(w)}\Im(z)\bigr)\cdot s+O(s^2)\quad\textrm{as $s \rightarrow 0$.}
\end{multline}
As a function of $z$, $H(z,w):=H_{\Gamma}(z,w)$ is  holomorphic on $\mathbb{H}$ and uniquely determined up to multiplication by a complex constant of absolute value one;
in addition, $H(z,w)$ is an automorphic form with a non-trivial multiplier system, which depends on $w$, with respect to
$\Gamma$ acting on $z$.  The function $H(z,w)$ vanishes if and only if $z=\gamma w$ for some $\gamma\in\Gamma$.
We call the function
$$
E_w(z)=E_{w,\Gamma}(z):=\log\bigl(|H(z,w)|^{\mathrm{ord}(w)}\Im(z)\bigr)
$$
the elliptic Kronecker limit function of $\Gamma$ at $w$.

\subsection{A factorization theorem}

We can now prove equation \eqref{ell kroneck limit one cusp2}.

\begin{proposition} \label{prop: factorization}
With notation as above, let
$f$ be a weight $2k$ meromorphic form on $\h$ with $q$-expansion at $\infty$ given by
\begin{equation} \label{q exp. of f_2k}
f(z)= 1+ \sum_{n=1}^{\infty} b_{f}(n)q_z^n,
\end{equation}
Let $\mathrm{ord}_w(f)$ denote the order $f$ at $w$ and define the function
$$
H_{f}(z):= \prod_{w \in \mathcal{F}} H(z,w)^{\mathrm{ord}_w(f)}
$$
where $H(z,w)=H_{\Gamma}(z,w)$ is given in  \eqref{Kronecker_elliptic}.  Then there exists a complex constant $c_{f}$ such that
\begin{equation} \label{factorization fla}
f(z) = c_{f}H_{f}(z).
\end{equation}
Furthermore,
$$
\abs{c_{f}} = \exp \left(-\frac{2\pi}{\vol_{\hyp}(M)} \sum_{w\in \mathcal{F}} \frac{\mathrm{ord}_w(f)}{\mathrm{ord}(w)}
\left( 2-\log 2 + P(w)- \beta\vol_{\hyp}(M)\right) \right ),
$$
where $P(w)$ and $\beta$ are defined through the parabolic Kronecker limit function
\eqref{KronLimitPArGen}.
\end{proposition}

\begin{proof}
The proof closely follows the proof of Theorem 9 from \cite{JvPS19}.
Specifically, following the first part of the proof almost verbatim, we conclude that the quotient
$$
F_f(z):=\frac{H_f(z)}{f(z)}
$$
is a non-vanishing holomorphic function on $M$ which is bounded and non-zero at the cusp at $\infty$.  Hence,
$\log \vert F_{f}(z)\vert$ is $L^{2}$ on $M$. From its spectral expansion and the fact that $\log \vert F_{f}(z)\vert$
is harmonic, one concludes $\log \vert F_{f}(z)\vert$ is constant, hence so is $F_{f}(z)$. The
evaluation of the constant is obtained by considering the limiting behavior as $z$ approaches $\infty$,
which is obtained by using the
asymptotic behavior of $H(z,w)$ as $\Im(z)\to\infty$, as given in Proposition 6 of \cite{JvPS19}.
\end{proof}

By following the proof of Proposition 12 from \cite{JvPS19} we obtain \eqref{ell kroneck limit one cusp}, and hence \eqref{ell kroneck limit one cusp2}, for meromorphic
forms $f$ on $\h$ with $q$-expansion \eqref{q exp. of f_2k}.  We leave the verification of this simple argument to the reader.

\section{Proofs of main results}

\subsection{Proof of Theorem \ref{thm:main}}

Let $Y>1$ be sufficiently large so that the cuspidal neighborhood $\mathcal{F}_{\infty}(Y)$ of the cusp $\infty$ in $\mathcal{F}$ is of
the form $\{z \in \HH : 0 < x < 1, y > Y\}$.  For $s\in\C$ with $\Re(s)>1$, and arbitrary, but fixed $w\in\mathcal{F}$, we then have that
\begin{align*}
  \int\limits_{\mathcal{F}(Y)}\Delta_{\hyp}(F_{-n}(z,1))&\left(G_s(z,w) + \E_\infty^{\mathrm{par}}(w,s)\right)d\mu_{\hyp}(z) \\ &
  -\int\limits_{\mathcal{F}(Y)}F_{-n}(z,1)\Delta_{\hyp} \left(G_s(z,w) + \E_\infty^{\mathrm{par}}(w,s)\right)d\mu_{\hyp}(z) \\ &= -s(1-s)\int\limits_{\mathcal{F}(Y)}F_{-n}(z,1)\left(G_s(z,w) + \E_\infty^{\mathrm{par}}(w,s)\right)d\mu_{\hyp}(z).
\end{align*}
Actually, the first summand on the left-hand side is zero since $F_{-n}(n,1)$ is holomorphic; however, this
judicious form of the number zero is significant since we will use the method behind the Maass-Selberg theorem to
  study the left-hand side of the above equation. Before this, note that the integrand on the right-hand side of the above equation is
holomorphic at $s=1$.  As a result, we can write
\begin{align*}
\frac{\partial}{\partial s}&\left.\left(-s(1-s)\int\limits_{\mathcal{F}(Y)}F_{-n}(z,1)\left(G_s(z,w) +
\E_\infty^{\mathrm{par}}(w,s)\right)d\mu_{\hyp}(z)\right)\right|_{s=1}\\&=
\int\limits_{\mathcal{F}(Y)}F_{-n}(z,1)\lim_{s\to 1}\left(G_s(z,w) + \E_\infty^{\mathrm{par}}(w,s)\right)d\mu_{\hyp}(z).
\end{align*}
Therefore,
\begin{align}\label{M-S starting f-la}
\langle F_{-n}(z,1), &\overline{\lim_{s\to 1} \left(G_s(z,w) + \E_\infty^{\mathrm{par}}(w,s)\right)} \rangle \notag \\&=\lim_{Y\to \infty} \int\limits_{\mathcal{F}(Y)}F_{-n}(z,1)\lim_{s\to 1} \left(G_s(z,w) + \E_\infty^{\mathrm{par}}(z,s)\right)d\mu_{\hyp}(z) \notag \\&=
\lim_{Y\to \infty} \left[\frac{\partial}{\partial s}\left(  \int\limits_{\mathcal{F}(Y)}\Delta_{\hyp}(F_{-n}(z,1))\left(G_s(z,w) +
\E_\infty^{\mathrm{par}}(w,s)\right)d\mu_{\hyp}(z)  \right. \right.\notag \\&-
\left.\left. \left. \int\limits_{\mathcal{F}(Y)}F_{-n}(z,1)\Delta_{\hyp} \left(G_s(z,w) + \E_\infty^{\mathrm{par}}(w,s)\right)d\mu_{\hyp}(z)  \right) \right|_{s=1}\right]
\end{align}

The quantity on the right-hand side of \eqref{M-S starting f-la} is setup for an application of Green's theorem as
in the proof of the Maass-Selberg relations for the Eisenstein series.  As described on page 89 of \cite{Iwa02}, when
applying Green's theorem to each term on the right-side of \eqref{M-S starting f-la} for fixed $Y$, the resulting
boundary terms on the sides of the fundamental domain, which are identified by $\Gamma$, will sum to zero.  As such,
we get that
\begin{align}\label{M-S interm f-la}
\langle F_{-n}(z,1),& \overline{\lim_{s\to 1} \left(G_s(z,w) + \E_\infty^{\mathrm{par}}(w,s)\right)} \rangle \notag \\&=
\lim_{Y\to \infty} \left[\frac{\partial}{\partial s}\left(  \int\limits_{0}^{1}\frac{\partial}{\partial y}F_{-n}(z,1)\left(G_s(z,w) +
\E_\infty^{\mathrm{par}}(w,s)\right)dx \right. \right. \notag\\ &-
\left.\left. \left. \int\limits_{0}^{1}F_{-n}(z,1)\frac{\partial}{\partial y} \left(G_s(z,w) + \E_\infty^{\mathrm{par}}(w,s)\right)dx
\right) \right|_{s=1}\right],
\end{align}
where functions of $z$ and its derivatives with respect to $y=\Im(z)$ are evaluated at $z=x+iY$.

In order to compute the difference of the two integrals of the right-hand side of \eqref{M-S interm f-la}, we will use the Fourier expansions
\eqref{Four exp Nieb at 1} and \eqref{Four exp Green} of the series $F_{-n}(z,1)$ and $G_s(z,w)$ respectively. It will be more convenient to write the first coefficient in the expansion \eqref{Four exp Nieb at 1} as $e(-nx)\sqrt{y}I_{\tfrac{1}{2}}(2\pi n y)$, as in \eqref{Four exp Nieb}.

Specifically, since the exponential functions $e(-nx)$ are orthogonal for different values of $n$, we get that
\eqref{M-S interm f-la} is equal to
\begin{multline*}
-F_{-n}(w,s)\sqrt{Y} \left(\frac{\partial}{\partial y}\left.\left(\sqrt{y}I_{\tfrac{1}{2}}(2\pi n y)\right)\right|_{y=Y}\cdot
K_{s-\tfrac{1}{2}}(2\pi n Y) \right. \\ \left.- I_{\tfrac{1}{2}}(2\pi n Y)\cdot \frac{\partial}{\partial y}\left.\left(\sqrt{y}K_{s-
\tfrac{1}{2}}(2\pi n y) \right)\right|_{y=Y}\right)
\end{multline*}
\begin{multline*}
+B_0(1;-n)(1-s)\frac{Y^{-s}}{2s-1}\E_\infty^{\mathrm{par}}(w,s) \\+
\sum_{j\in\Z\smallsetminus\{0\}}F_j(w,s) \left( b_j(Y,1;-n)\cdot \frac{\partial}{\partial y}\left.\left(\sqrt{y}K_{s-
\tfrac{1}{2}}(2\pi |j| y) \right)\right|_{y=Y} \right. \\ \left.-\frac{\partial}{\partial y}\left.b_j(y,1;-n)\right|_{y=Y}\cdot \sqrt{Y}
K_{s-\tfrac{1}{2}}(2\pi |j| Y) \right) =T_1(Y,s;w) + T_2(Y,s;w)+T_3(Y,s;w),
\end{multline*}
where the last equality above provides the definitions of the functions $T_{1}$, $T_{2}$ and $T_{3}$.  Therefore, from
\eqref{M-S interm f-la} we conclude that
\begin{multline}\label{Three terms pre-comp}
  \langle F_{-n}(z,1), \overline{\lim_{s\to 1} \left(G_s(z,w) + \E_\infty^{\mathrm{par}}(w,s)\right)} \rangle =\\= \lim_{Y\to \infty}
  \left[\left.\frac{\partial}{\partial s}\left( T_1(Y,s;w) + T_2(Y,s;w)+T_3(Y,s;w) \right)\right|_{s=1}\right]
\end{multline}
We will treat each of the three terms on the right-hand side of \eqref{Three terms pre-comp} separately.

To evaluate the term $T_{1}$ in \eqref{Three terms pre-comp}, we apply
formulas 8.486.2 and 8.486.11 of \cite{GR07} in order to compute derivatives of the Bessel functions.
In doing so, we conclude that
$$
T_1(Y,s;w)=-\frac{X}{2}F_{-n}(w,s)\left[K_{s-\tfrac{1}{2}}(X)(I_{-\tfrac{1}{2}}(X) + I_{\tfrac{3}{2}}(X)) + I_{\tfrac{1}{2}}(X)(K_{s-\tfrac{3}{2}}(X) + K_{s+\tfrac{1}{2}}(X)) \right],
$$
where we set $X=2\pi n Y$.  Next, we express $K_{s+\tfrac{1}{2}} (X)$ in terms of $K_{s-\tfrac{1}{2}} (X)$ and $K_{s-\tfrac{3}{2}}(X)$, using formula 8.485.10 from \cite{GR07} to get
$$
K_{s+\tfrac{1}{2}} (X)= K_{s-\tfrac{3}{2}} (X) + \frac{2s-1}{X}K_{s-\tfrac{1}{2}} (X).
$$
Then, applying formula 8.486.21 from \cite{GR07}, we deduce that
\begin{align*}
\frac{\partial}{\partial s}&\left.\left[K_{s-\tfrac{1}{2}}(X)(I_{-\tfrac{1}{2}}(X) + I_{\tfrac{3}{2}}(X)) + I_{\tfrac{1}{2}}(X)(K_{s-\tfrac{3}{2}}(X) + K_{s+\tfrac{1}{2}}(X)) \right]\right|_{s=1}\\&= \sqrt{\frac{\pi}{2X}}e^X\mathrm{Ei}(-2X)\left[ -(I_{-\tfrac{1}{2}}(X) + I_{\tfrac{3}{2}}(X)) + \sqrt{\frac{2}{\pi X}} (2-1/X) \sinh (X) \right]+\frac{2}{X^2}e^{-X}\sinh (X),
\end{align*}
where $\mathrm{Ei}(x)$ denotes the exponential integral; see section 8.2 of  \cite{GR07}.  Continuing, we now employ
formula (B.36) from \cite{Iwa02} which asserts certain asymptotic behavior of the $I$-Bessel function as $X\to\infty$;
we are interested in the cases when $\nu=-1/2$ and when $\nu=3/2$.
This result, together with the bound $\mathrm{Ei}(-2X)\leq e^{-2X}/(2X) $, which follows from the expression 8.212.10 from \cite{GR07} for $\mathrm{Ei}(-x)$ with $x>0$, yields that
$$
\lim_{X\to\infty}\frac{X}{2}\frac{\partial}{\partial s}\left.\left[K_{s-\tfrac{1}{2}}(X)(I_{-\tfrac{1}{2}}(X) +
I_{\tfrac{3}{2}}(X)) + I_{\tfrac{1}{2}}(X)(K_{s-\tfrac{3}{2}}(X) + K_{s+\tfrac{1}{2}}(X)) \right]\right|_{s=1} =0.
$$
Therefore,
\begin{multline*}
\lim_{Y\to \infty}\frac{\partial}{\partial s}\left. T_1(Y,s;w)\right|_{s=1}= - \frac{\partial}{\partial s}\left. F_{-n}(w,s)\right|_{s=1}\cdot \\\cdot\lim_{X\to \infty}\frac{X}{2}\left[K_{s-\tfrac{1}{2}}(X)(I_{-\tfrac{1}{2}}(X) + I_{\tfrac{3}{2}}(X)) + I_{\tfrac{1}{2}}(X)(K_{s-\tfrac{3}{2}}(X) + K_{s+\tfrac{1}{2}}(X)) \right].
\end{multline*}
Finally, by applying (B.36) from \cite{Iwa02} again, we deduce that
$$
\lim_{X\to \infty}\frac{X}{2}\left[K_{s-\tfrac{1}{2}}(X)(I_{-\tfrac{1}{2}}(X) + I_{\tfrac{3}{2}}(X)) + I_{\tfrac{1}{2}}(X)
(K_{s-\tfrac{3}{2}}(X) + K_{s+\tfrac{1}{2}}(X)) \right]=1.
$$
Hence
\begin{equation}\label{term1}
\lim_{Y\to \infty}\frac{\partial}{\partial s}\left. T_1(Y,s;w)\right|_{s=1}= - \frac{\partial}{\partial s}\left. F_{-n}(w,s)\right|_{s=1}.
\end{equation}

As for the term $T_{2}$ in \eqref{Three terms pre-comp}, let us use
the Laurent series expansion \eqref{KronLimitPArGen} of $\E_\infty^{\mathrm{par}}(w,s)$, from
which one easily deduces that
$$
\frac{\partial}{\partial s}\left.(s-1)\frac{Y^{-s}}{2s-1}\E_\infty^{\mathrm{par}}(w,s)\right|_{s=1} = \frac{1}{Y}\left( \beta -
\frac{P(w)+2+\log Y}{\vol_{\hyp}(M)}\right).
$$
Therefore
\begin{equation}\label{term2}
\lim_{Y\to \infty}\frac{\partial}{\partial s}\left. T_2(Y,s;w)\right|_{s=1}=0.
\end{equation}

It remains to study the term $T_{3}$ in \eqref{Three terms pre-comp}.
Let us set $g(s,y,k):=\sqrt{y}K_{s-\tfrac{1}{2}}(2\pi k y)$ for some positive integers $k$. Then $b_j(y,1;-n)=B_j(1;-n)g(1,y,|n|)$ and
\begin{multline*}
T_3(Y,s;w)=\sum_{j\in\Z\smallsetminus\{0\}} B_j(1;-n) F_j(w,s) \left( g(1,Y,|n|) \frac{\partial}{\partial y}\left.g(s,y,|j|)\right|_{y=Y} \right. \\ - \left. g(s,Y;|j|)\frac{\partial}{\partial y}\left.g(1,y,|n|)\right|_{y=Y} \right).
\end{multline*}

For positive integers $m$ and $\ell$ let us define
$$
G(s,Y,m,\ell):=g(1,Y,m) \frac{\partial}{\partial y}\left.g(s,y,\ell)\right|_{y=Y}- g(s,Y;\ell)\frac{\partial}{\partial y}\left.g(1,y,m)\right|_{y=Y}.
$$
Applying the formula 8.486.11 from \cite{GR07} to differentiate the $K-$Bessel function, together with formula 8.486.10 to express $K_{s+\tfrac{1}{2}}(2\pi |j|Y)$ we arrive at
\begin{multline*}
G(s,Y,|n|,|j|) =\frac{\pi Y}{2}K_{s-\tfrac{1}{2}}(2\pi |j| Y)K_{\tfrac{1}{2}}(2\pi |n| Y) \cdot \\ \cdot \left( |n|(K_{-\tfrac{1}{2}}(2\pi |n|Y) + K_{\tfrac{3}{2}}(2\pi |n|Y)) - |j|\left(2K_{s-\tfrac{3}{2}}(2\pi |j|Y) +\frac{2s-1}{2\pi |j| Y} K_{s-\tfrac{1}{2}}(2\pi |j|Y)\right)\right).
\end{multline*}

Now, we combine the bound (B.36) from \cite{Iwa02} with evaluation of the derivative $\frac{\partial}{\partial \nu} K_{\nu}$ at $\nu=\pm 1/2$ (formula 8.486.21 of \cite{GR07}) and the bound $\mathrm{Ei}(-4\pi |j|Y) \leq \exp(-4\pi |j|Y)/(4\pi |j|Y)$ for the exponential integral function to deduce the following crude bounds
$$
\max\left\{G(s,Y,|n|,|j|), \left.\frac{\partial }{\partial s} G(s,Y,|n|,|j|)\right|_{s=1}\right\}\ll (|n| + |j|)\exp(-2\pi Y (|n|+|j|)), \text{   as   } Y\to +\infty,
$$
where the implied constant is independent of $Y,|j|$.

This, together with the bound \eqref{eq: N-P deriv bound} and the Fourier expansion \eqref{Four exp Nieb at 1} yields
$$
\frac{\partial}{\partial s} T_3(Y,s;w)\Big|_{s=1}\ll
\sum_{j\in\Z\smallsetminus\{0\}}(|n| + |j|) | B_j(1;-n)| \exp\left( -2\pi Y (|n|+|j|) + 2\pi|j| \Im(w)\right)
$$
It remains to estimate the sum on the right hand  side of the above equation as $Y\to\infty$.
The bounds for the Kloosterman sum zeta function  as stated on page 75 of \cite{Iwa02})
yield bounds for $B_j(1;-n)$ for $j\neq 0$.  Specifically, one has that
$$
B_j(1;-n)\ll \exp\left(\frac{4\pi \sqrt{|jn|}}{c_\Gamma}\right)
$$
where $c_\Gamma$ is a certain positive constant depending on the group $\Gamma$; in fact, $c_{\Gamma}$ is equal to the minimal positive
left-lower entry of a matrix from $\Gamma$.  Also, the implied constant in the bound for $B_j(1;-n)$ is independent of $j$. Therefore
$$
\frac{\partial}{\partial s}\left. T_3(Y,s;w)\right|_{s=1}\ll \sum_{j\in\Z\smallsetminus\{0\}} (|n| + |j|)\exp\left( -2\pi \left( (|j|+|n|)Y - 2\sqrt{|jn|}/c_\Gamma -|j| \Im(w) \right)  \right).
$$
For $Y > 2\Im(w) + 2\sqrt{n}/c_\Gamma$, this series over $j$ is uniformly convergent and is $o(1)$ as $Y\to\infty$.
In other words,
\begin{equation}\label{T3_limit}
\lim_{Y\to\infty}\frac{\partial}{\partial s}\left. T_3(Y,s;w)\right|_{s=1} =0.
\end{equation}
When combining \eqref{T3_limit} with  \eqref{Three terms pre-comp} \eqref{term1} and \eqref{term2},
we have that
\begin{equation*}
  \langle F_{-n}(z,1), \overline{\lim_{s\to 1} \left(G_s(z,w) + \E_\infty^{\mathrm{par}}(w,s)\right)} \rangle = \frac{\partial}{\partial s}\left. F_{-n}(w,s) \right|_{s=1},
\end{equation*}
which completes the proof of \eqref{main f-la}.

\subsection{Proof of Corollary \ref{Rohr-Jensen}}

The proof of Corollary \ref{Rohr-Jensen} is a combination of Theorem \ref{thm:main} and the factorization theorem as stated in
Proposition \ref{prop: factorization}.  The details are as follows.

To begin we shall prove formula \eqref{log norm basic}.  Starting with \eqref{ell kroneck limit one cusp2}, which is written as
$$
\log\left( y^k |f(z)| \right) = kP(z) -  \sum_{w\in \mathcal{F}} \frac{\mathrm{ord}_w(f)}{\mathrm{ord}(w)} \lim_{s\to 0} \frac{1}{s} \mathrm{ord}(w)\mathcal{E}^{\mathrm{ell}}_{w}(z,s),
$$
we can express $\lim_{s\to 0} \frac{1}{s}\mathrm{ord}(w) \mathcal{E}^{\mathrm{ell}}_{w}(z,s)$ in terms of the resolvent kernel.
Specifically, using \eqref{Ell Green connection}, we have that
\begin{equation} \label{ell kroneck limit one cusp3}
\log\left( y^k |f(z)| \right)=kP(z) + \sum_{w\in \mathcal{F}} \frac{\mathrm{ord}_w(f)}{\mathrm{ord}(w)} \lim_{s\to 0}
\left(\frac{2^{s}\sqrt{\pi} \Gamma(s-1/2)}{\Gamma(s+1)}(2s-1)G_s(z,w)\right).
\end{equation}
By applying the functional equation for the Green's function, see Theorem 3.5 of \cite{He83} on pages 250--251, we get
\begin{align*}
\lim_{s\to 0}\frac{2^{s}\sqrt{\pi} \Gamma(s-1/2)}{\Gamma(s+1)}(2s-1)G_s(z,w) =
\lim_{s\to 1}&\left(\frac{2^{1-s}\sqrt{\pi} \Gamma(1/2-s)}{\Gamma(2-s)}\left((1-2s)G_s(z,w)\right.\right.
\\&- \left.\left.\frac{}{}\E_\infty^{\mathrm{par}}(z,1-s)\E_\infty^{\mathrm{par}}(w,s)\right)\right).
\end{align*}
From the Kronecker limit formula \eqref{KronLimas s to 0} and standard Taylor series expansion of the gamma function we immediately deduce that
\begin{multline*}
\lim_{s\to 0}\frac{2^{s}\sqrt{\pi} \Gamma(s-1/2)}{\Gamma(s+1)}(2s-1)G_s(z,w)= \lim_{s\to 1}2\pi (-1+(s-1)(2-\log 2))\cdot \\ \cdot \left[2(1-s)G_s(z,w) - (G_s(z,w)+\E_\infty^{\mathrm{par}}(w,s)) - P(z)(1-s)\E_\infty^{\mathrm{par}}(w,s)\right].
\end{multline*}
According to \cite{Iwa02}, p. 106, the point $s=1$ is the simple pole of $G_s(z,w)$ with the residue $-1/\vol_{\hyp}(M)$ (note that our $G_s(z,w)$ differs from the automorphic Green's function from \cite{Iwa02} by a factor of $-1$). Therefore, the Kronecker limit formula \eqref{KronLimitPArGen} yields the following equation
\begin{align} \label{Limit of G as s to 0}
\lim_{s\to 0}\frac{2^{s}\sqrt{\pi} \Gamma(s-1/2)}{\Gamma(s+1)}(2s-1)G_s(z,w)&= -\frac{2\pi}{\vol_{\hyp}(M)}P(z)- \frac{4\pi}{\vol_{\hyp}(M)} \\&
+ 2\pi\lim_{s\to 1}\left(G_s(z,w) + \E_\infty^{\mathrm{par}}(w,s)\right).\notag
\end{align}
Recall that the classical the Riemann-Roch theorem implies that
$$
k\frac{\vol_{\hyp}(M)}{2\pi}=\sum_{w\in \mathcal{F}} \frac{\mathrm{ord}_w(f)}{\mathrm{ord}(w)};
$$
hence, after multiplying \eqref{Limit of G as s to 0} by $\frac{\mathrm{ord}_w(f)}{\mathrm{ord}(w)} $ and taking the sum over all $w\in\mathcal{F}$
from \eqref{ell kroneck limit one cusp3}, we arrive at \eqref{log norm basic}, as claimed.

Having proved \eqref{log norm basic}, observe that the left-hand side of \eqref{log norm basic} is real valued.
As proved in \cite{Ni73}, $F_{-n}(z,1)$ is orthogonal to constant functions.  Therefore, in order to prove \eqref{main f-la - corollary}
one simply applies \eqref{main f-la}, which was established above.

\subsection{Proof of Corollary \ref{cor:j-function}}
In order to prove \eqref{main f-la - corollary 2}, it suffices to compute
$\langle 1, \overline{\lim_{s\to 1}(G_s(z,w) + \E_{\infty}^{\mathrm{par}}(w,s))} \rangle $, which we will
write as
$$
\int\limits_{\mathcal{F}}\lim_{s\to 1}\left(G_s(z,w) + \frac{1}{\vol_{\hyp}(M)(s-1)}+ \E_{\infty}^{\mathrm{par}}(w,s)-\frac{1}{\vol_{\hyp}(M)(s-1)}\right)d\mu_{\hyp}(z).
$$
From its spectral expansion, the function $\lim_{s\to 1}\left(G_s(z,w) + \frac{1}{\vol_{\hyp}(M)(s-1)}\right)$ is $L^2$ on $\mathcal{F}$ and orthogonal to constant functions.
Therefore, by using the Laurent series expansion \eqref{KronLimitPArGen}, we get that
$$
\langle 1, \overline{\lim_{s\to 1}(G_s(z,w) + \E_{\infty}^{\mathrm{par}}(w,s))} \rangle =\vol_{\hyp}(M)\left(\beta - \frac{P(w)}{\vol_{\hyp}(M)}\right),
$$
which completes the proof.

\subsection{Proof of Theorem \ref{thm:generating series}}\label{sect proof of thm 4}

Our starting point is the Fourier expansion of the sum $G_s(z,w) + \E_\infty^{\mathrm{par}}(w,s)$.
Namley, for $\Re(s)>1$ and $\Im(w)$ sufficiently large we have that
\begin{align}\label{G+Epar}
G_s(z,w) + \E_\infty^{\mathrm{par}}(w,s)&=\left(1-\frac{y^{1-s}}{2s-1}\right)\E_\infty^{\mathrm{par}}(w,s) \notag
\\&-\sum_{k\in\mathbb{Z}\setminus\{0\}} \sqrt{y}K_{s-\tfrac{1}{2}}(2\pi |k| y)F_{-k}(w,s)e(kx).
\end{align}
If $\Im(z)$ is sufficiently large, exponential decay of $K_{s-\tfrac{1}{2}}(2\pi |k| y)$ is sufficient to ensure that the right-hand side of \eqref{G+Epar} is holomorphic at $s=1$. The Laurent series expansion of $\E_\infty^{\mathrm{par}}(w,s)$, combined with the expansions $y^{1-s}=1+(1-s)\log y + \tfrac{1}{2}(1-s)^2 \log^2y + O((1-s)^3)$ and $(2s-1)^{-1}= (1-2(s-1))^{-1} = 1-2(s-1)+4(s-1)^2 + O((s-1)^3)$ yields
\begin{multline*}
\frac{\partial}{\partial s}\left.\left(1-\frac{y^{1-s}}{2s-1}\right)\E_\infty^{\mathrm{par}}(w,s) \right|_{s=1} =\frac{1}{\vol_{\hyp}(M)}\left[ -4+2\beta\vol_{\hyp}(M)-2P(w)\right. \\ \left. +
\log y \left(\beta\mathrm{vol}_{\hyp}(M) - P(w)-2\right) - \tfrac{1}{2}\log^2y \right].
\end{multline*}
Additionally, for $\Im(z)$ sufficiently large, the series on the right-hand side of \eqref{G+Epar} is a uniformly convergent series of functions
which are holomorphic at $s=1$.  As such, we may differentiate the series  term by term.  By employing formulas 8.469.3 and 8.486.21 of \cite{GR07},
we deduce for $k\neq0$ that
\begin{multline*}
\frac{\partial}{\partial s}\left. \left(\sqrt{y}K_{s-\tfrac{1}{2}}(2\pi |k| y)F_{-k}(w,s) \right) \right|_{s=1}=\frac{e^{-2\pi|k|y}}{2\sqrt{|k|}}\cdot \\ \cdot\left[ \frac{\partial}{\partial s}\left. F_{-k}(w,s) \right|_{s=1}- F_{-k}(w,1)e^{4\pi|k|y} \mathrm{Ei}(-4\pi |k|y)\right],
\end{multline*}
where $\mathrm{Ei}(x)$ denotes the exponential integral function; see section 8.21 of \cite{GR07}.  From this, we get the
expression that
\begin{multline*}
\frac{\partial}{\partial s}\left(G_s(z,w) + \E_\infty^{\mathrm{par}}(w,s)\right)\Big|_{s=1} = (\log y+2) \left(\beta - \frac{P(w)+2}
{\mathrm{vol}_{\hyp}(M)}\right) - \frac{\log^2 y}{2\mathrm{vol}_{\hyp}(M)} \\- \sum_{k\in\mathbb{Z}\setminus\{0\}}\frac{1}{2\sqrt{|k|}}\left[ \frac{\partial}{\partial s}
\left. F_{-k}(w,s) \right|_{s=1}- F_{-k}(w,1)e^{4\pi|k|y} \mathrm{Ei}(-4\pi |k|y)\right]e^{2\pi i kx - 2\pi|k|y}.
\end{multline*}
Let us now compute the derivative $\frac{\partial}{\partial z}$ of the above expression.
After multiplying by $i=\sqrt{-1}$, we get that
\begin{multline*}
\mathcal{G}_w(z)=\frac{1}{y}\left(\beta - \frac{P(w)+2}{\mathrm{vol}_{\hyp}(M)}\right) -  \frac{\log y}{y \mathrm{vol}_{\hyp}(M)} +\sum_{k\geq 1} 2\pi \sqrt{k} \frac{\partial}{\partial s}\left. F_{-k}(w,s) \right|_{s=1}q_z^k  \\ +\sum_{k\geq 1}\frac{F_{-k}(w,1)}{2\sqrt{k}y}q_z^k -\sum_{k\leq -1} 2\pi \sqrt{|k|} F_{-k}(w,1)
\mathrm{Ei}(4\pi ky) q_z^k + \sum_{k\leq -1}\frac{F_{-k}(w,1)}{2\sqrt{|k|}y}e^{2\pi i k(x-iy)}.
\end{multline*}
The proof of the assertion that $\sum_{k\geq 1} 2\pi \sqrt{k} \frac{\partial}{\partial s}\left. F_{-k}(w,s) \right|_{s=1}q_z^k$ is the holomorphic part of $\mathcal{G}_w(z)$ follows by citing the uniqueness of the analytic continuation in $z$.

It is left to prove that $\mathcal{G}_w(z)$ is weight two biharmonic Maass form. Since $\mathcal{G}_w(z)$ is obtained by taking the derivative $\frac{\partial}{\partial z}$ of a $\Gamma-$invariant function, it is obvious that $\mathcal{G}_w(z)$ is weight two in $z$. Moreover, the straightforward computation that
$$
iy^2 \frac{\partial}{\partial \bar z}\mathcal{G}_w(z)=\Delta_{\hyp}\left(\frac{\partial}{\partial s}\left(G_s(z,w) + \E_\infty^{\mathrm{par}}(w,s)\right)\Big|_{s=1}\right)=-
 \lim_{s\to 1} \left(G_s(z,w) + \E_\infty^{\mathrm{par}}(w,s)\right),
 $$
combined with the fact that $\Delta_{\hyp}\left( \lim_{s\to 1} \left(G_s(z,w) + \E_\infty^{\mathrm{par}}(w,s)\right)\right)=0$ proves that $\mathcal{G}_w(z)$ is biharmonic.

\section{Examples }\label{sect:examples}

\subsection{The full modular group}

Throughout this subsection, let $\Gamma=\mathrm{PSL}(2,\Z)$, in which case the the parabolic Kronecker limit function, $P(w)$ can be expressed, in the notation of \cite{BK19}, as $$P(w)=P_{\mathrm{PSL}(2,\mathbb{Z})}(w)=\log(|\eta(w)|^4 \cdot \Im(w))=\mathbbm{j}(w)-1,$$
where $\eta(w)$ is Dedekind's eta function and the last equality follows from the definition of $\mathbbm{j}_0(w)=\mathbbm{j}(w)$ given on p. 1 of \cite{BK19}.

In this setting, Corollary \ref{Rohr-Jensen}, when combined with \eqref{j_via_F} and Rohrlich's theorem
\eqref{rohrl thm} yields that
\begin{equation}\label{prod with jn}
\langle j_n,\log||f||\rangle =2\pi\sqrt{n}\left( - 2\pi \sum_{w\in \mathcal{F}} \frac{\mathrm{ord}_w(f)}{\mathrm{ord}(w)}\left(\frac{\partial}{\partial s}\left. F_{-n}(w,s) \right|_{s=1} -c_nP(w)\right)\right).
\end{equation}
Moreover, equating the constant terms in the Fourier series expansions for $F_{-n}(z,1)$ and $j_n(z)$, one easily deduces that $2\pi\sqrt{n}c_n=24\sigma(n)$.
This proves Theorem 1.2 of \cite{BK19} and shows that, in the notation of \cite{BK19} one has
\begin{equation}\label{jn for BK}
\mathbbm{j}_n(w)=2\pi\sqrt{n}\frac{\partial}{\partial s}\left. F_{-n}(w,s) \right|_{s=1} -24\sigma(n)P(w),
\end{equation}
an identity which provides a description of $\mathbbm{j}_n(w)$, for $n\geq 1$ different from the one given by formula (3.10) of \cite{BK19}.

Furthermore, from the identity \eqref{delta of deriv of N-P}, combined with the fact that $ \Delta_{\hyp}P(w)=1$, which is a straightforward implication of the Kronecker limit formula \eqref{KronLimitPArGen}, it follows that $$\Delta_{\hyp}\mathbbm{j}_n(w)= 2\pi\sqrt{n}\left( F_{-n}(w,1)-c_n \right)=j_n(w),$$
which agrees with formula (3.10) of \cite{BK19}.

Reasoning as above, we easily see that Theorem 1.3. of \cite{BK19} follows from Corollary \ref{cor:j-function} with $g(z)=j_n(z)$.

Finally, in view of \eqref{prod with jn}, Theorem \ref{thm:generating series} is closely related to the first part of Theorem 1.4 of \cite{BK19}. Namely, for large enough $\Im(z)$, in the notation of \cite{BK19}
\begin{align*}
\mathbb{H}_w(z)&=\sum_{n\geq 0} \mathbbm{j}_n(w) q_z^n=\mathbbm{j}_0(w)+\sum_{n\geq 1}\left(2\pi\sqrt{n}\frac{\partial}{\partial s}\left. F_{-n}(w,s)
\right|_{s=1} -24\sigma(n)P(w)\right)q_z^n\\&=1+P(w)\left(1-24\sum_{n\geq 1}\sigma(n)q_z^n\right) + \sum_{n\geq 1} 2\pi \sqrt{n} \frac{\partial}{\partial s}
\left. F_{-n}(w,s) \right|_{s=1}q_z^n.
\end{align*}

Theorem \ref{thm:generating series} implies that the function $\mathbb{H}_w(z)$ is the holomorphic part of the weight two biharmonic Maass form $$\widehat{\mathbb{H}}_w(z)=P(w)\widehat{E}_2(z)+\mathcal{G}_w(z),$$ where
$$
\widehat{E}_2(z)=1-24\sum_{n\geq 1}\sigma(n)q_z^n - \frac{3}{\pi y}
$$
is the weight two completed Eisenstein series for the full modular group.


\subsection{Genus zero Atkin-Lehner groups}

Let $N=\prod_{\nu=1}^r p_\nu$ be a positive square-free integer which is one of the $44$ possible values for which the quotient
space $Y_{N}^{+} =\overline{ \Gamma_{0}^{+}(N)}\backslash \HH$ has genus zero; see \cite{Cum04} for a list of
such $N$ as well as \cite{JST14}.  Let $\Delta_{N}(z)$ be the Kronecker limit function on $Y_{N}^{+}$ associated to the
parabolic Eisenstein series; it is given by formula \eqref{DeltaN} above.

In the notation of Section \ref{sect Atkin Leh groups_KLF},the function $\Delta_N(z)(j_N^+(z)-j_N^+(w))^{\nu_N}$, is the weight $k_N=2^{r-1}\ell_N$ holomorphic modular form which possesses the constant term $1$ in its $q$-expansion.
Furthermore, this function vanishes only at the point $z=w$, and, by the Riemann-Roch formula, its order of vanishing is equal to $k_N \vol_{\hyp}(Y_N^{+})\cdot\mathrm{ord}(w)/(4\pi )$.

When $N=1$, one has $k_1=12$, $\ell_1=24$, $\nu_1=1$ and $\vol_{\hyp}(Y_N^{+})=\pi/3$, hence $\Delta_1(z)(j_1^+(z)-j_1^+(w))^{\nu_1}$ equals the prime form $(\Delta(z)(j(z)-j(w)))^{1/\mathrm{ord}(w)}$ taken to the power $\mathrm{ord}(w)$; see page 3 of \cite{BK19}.

For any integer $m>1$ the $q$-expansion of the form $j_N^+|T_{m}(z)$ is $q_z^{-m}+O(q_z)$; hence there exists a constant $C_{m,N}$ such that
$j_N^+|T_{m}(z)=2\pi \sqrt{m} F_{-m}(z,1) + C_{m,N}$.
The constant $C_{m,N}$ can be explicitly evaluated in terms of $m$ and $N$ by equating the constant terms in the $q$-expansions.
Upon doing so, one obtains, using equation \eqref{B0 for A-L groups}, that
\begin{align*}
C_{m,N}=-2\pi \sqrt{m} B_{0,N}^+(1;-m)&=-24\sigma(m)\prod\limits_{\nu=1}^r \left(1-
\frac{p_\nu^{\alpha_{p_\nu}(m)+1}(p_\nu -1) }{\left(p_\nu^{\alpha_{p_\nu}(m)+1} - 1\right)(p_\nu+1)}\right)\\&=-24\sigma(m) \prod\limits_{\nu=1}^r \left(1- \kappa_{m}(p_\nu)\right),
\end{align*}
where we simplified the notation by denoting the second term in the product over $\nu$ by $\kappa_m(p_\nu)$.
We now can apply Corollary \ref{cor:j-function} with
$$
g(z)= j_N^+|T_{m}(z)=2\pi \sqrt{m} F_{-m}(z,1)-24\sigma(m)\prod\limits_{\nu=1}^r \left(1- \kappa_{m}(p_\nu)\right)
$$
and $f(z)=\Delta_N(z)(j_N^+(z)-j_N^+(w))^{\nu_N}$. Corollary \ref{cor:j-function}
becomes the statement that
\begin{align*}
\langle j_N^+|T_{m}(z), &\log(y^{\tfrac{k_N}{2}} |\Delta_N(z)(j_N^+(z)-j_N^+(w))^{\nu_N}|) \rangle
\\& = -k_N \vol_{\hyp}(Y_N^{+})\left[\pi\sqrt{m}\left.\frac{\partial}{\partial s}F_{-m}(w,s)\right|_{s=1} \right. \\&
\left. +12\sigma(m)\prod\limits_{\nu=1}^r
\left(1- \kappa_{m}(p_\nu)\right) \left(\beta_N\vol_{\hyp}(Y_N^{+})-\log\left( |\Delta_N(w)|^{2/k_N} \cdot \Im (w)\right) -2\right)\right],
\end{align*}
where $\beta_N$ is given by \eqref{betaN}.  In this form, we have obtained an alternate
proof and generalization of formula (1.2) from \cite{BK19}, which is the special case $N=1$.

\subsection{A genus one example}

Let us consider the case when $\Gamma = \overline{\Gamma_{0}(37)^{+}}$.  The choice of $N=37$ is significant since
this level corresponds to the smallest square-free integer $N$ such that
$Y_{N}^{+}$ is genus one.  From Proposition 11 of \cite{JST13}, we have that
$\vol_{\hyp}(Y_{37}^{+})= 19\pi/3$ and
$$
\beta_{37}=\frac{3}{19\pi}\left(\frac{10}{19}\log37+2-2\log(4\pi)-24\zeta'(-1)\right).
$$
The function field generators are $x_{37}^+(z)=q_z^{-2} + 2q_z^{-1}+ O(q_z)$ and $y_{37}^+(z)=q_z^{-3} + 3q_z^{-1}+ O(q_z)$, as displayed in Table 5 of \cite{JST13}. The generators $x_{37}^+(z)$ and $y_{37}^+(z)$ satisfy the cubic relation $y^2 - x^3 + 6xy - 6x^2 + 41y + 49x + 300 = 0$.

The functions $x_{37}^+(z)$ and $y_{37}^+(z)$ can be expressed in in terms of the Niebur-Poincar\'e series by comparing their
$q$-expansions.  The resulting expressions are that
\begin{align*}
x_{37}^+(z)&=2\pi[\sqrt{2}F_{-2}(z,1)+ 2F_{-1}(z,1)]-2\pi(\sqrt{2}B_{0,37}^+(1;-2)+2B_{0,37}^+(1;-1))\\&=2\pi[\sqrt{2}F_{-2}(z,1)+ 2F_{-1}(z,1)]-\frac{60}{19}
\end{align*}
and
\begin{align*}
y_{37}^+(z)&=2\pi[\sqrt{3}F_{-3}(z,1)+ 3F_{-1}(z,1)]-2\pi(\sqrt{3}B_{0,37}^+(1;-3)+3B_{0,37}^+(1;-1))\\&=2\pi[\sqrt{3}F_{-3}(z,1)+ +3F_{-1}(z,1)]-\frac{84}{19}.
\end{align*}

It is important to note that $x_{37}^+(z)$ has a pole of order two at $z=\infty$, i.e., its $q$-expansion begins with $q_z^{-2}$.
As such, $x_{37}^+(z)$ is a linear transformation of the
Weierstrass $\wp$-function, in the coordinates of the upper half plane, associated to the elliptic curve obtained by
compactifying the space $Y_{37}^{+}$.  Hence, there are three distinct points
$\{w\}$ on $Y_{37}^{+}$, corresponding to the two torsion points under the group law, such that $x_{37}^+(z)-x_{37}^+(w)$
vanishes as a function of $z$ only when $z=w$.  The order of vanishing necessarily is equal to two.  The cusp form
$\Delta_{37}(z)$ vanishes at $\infty$ to order $19$.  Therefore, for such $w$, the form
$$
f_{37,w}(z)=\Delta_{37}^2(z)(x_{37}^+(z)-x_{37}^+(w))^{19}
$$
is a weight $2k_{37}=24$ holomorphic form.  The constant term in its $q$-expansion is equal to $1$,
and $f_{37,w}(z)$ vanishes for points $z \in \mathcal{F}$ only when  $z=w$.  The order of vanishing of $f_{37,w}(z)$ at $z=w$ is $38\cdot \mathrm{ord}(w)$.

With all this, we can apply Corollary \ref{cor:j-function}. The resulting formulas are that
\begin{align*}
\langle x_{37}^+, \log(\Vert f_{37,w}\Vert ) \rangle &= -152\pi^2 \left(\frac{\partial}{\partial s}\left. (\sqrt{2}F_{-2}(w,s) +
2F_{-1}(w,s))\right|_{s=1}\right)\\& +240\pi\left(\log\left(|\eta(w)\eta(37w)|^2\cdot\Im (w)\right) -\frac{10}{19}\log37+2\log(4\pi) +24\zeta'(-1)\right)
\end{align*}
and
\begin{multline*}
\langle y_{37}^+, \log (\Vert f_{37,w}\Vert ) \rangle = - 152\pi^2 \left(\frac{\partial}{\partial s}\left. (\sqrt{3}F_{-3}(w,s) +3F_{-1}(w,s))\right|_{s=1}\right)\\ + 336\pi\left(\log\left(|\eta(w)\eta(37w)|^2\cdot\Im (w)\right) -\frac{10}{19}\log37+2\log(4\pi) +24\zeta'(-1)\right).
\end{multline*}

Of course, one does not need to assume that $w$ corresponds to a two torsion point.
In general, Corollary \ref{cor:j-function} yields an expression where the right-hand side
is a sum of two terms, and the corresponding factor in front would be one-half of the factors above.

\subsection{A genus two example}

Consider the level $N=103$.
In this case, $\vol_{\hyp}(Y_{103}^{+})= 52\pi/3$ and the function field generators are $x_{103}^+(z)=q_z^{-3} + q_z^{-1} + O(q_z)$
and $y_{103}^+(z)=q_z^{-4} + 3q_z^{-2} + 3q_z^{-1} + O(q_z)$, as displayed in Table 7 of \cite{JST13}. The generators $x_{103}^+(z)$ and
$y_{103}^+(z)$ satisfy the polynomial relation $y^3 - x^4 - 5yx^2 - 9x^3 + 16y^2 - 21yx - 60x^2 + 65y - 164x + 18 = 0$.
The surface $Y_{103}^{+}$ has genus two.

From Theorem 6 of \cite{Ni73}, we can write $x_{103}^+(z)$ and $y_{103}^+(z)$ in terms of the Niebur-Poincar\'e series.
Explictly, we have that
\begin{align*}
x_{103}^+(z)&=2\pi[\sqrt{3}F_{-3}(z,1)+ F_{-1}(z,1)]-2\pi(\sqrt{3}B_{0,103}^+(1;-3)+B_{0,103}^+(1;-1))\\&=2\pi[\sqrt{3}F_{-3}(z,1)+ F_{-1}(z,1)]-\frac{15}{13}
\end{align*}
and
\begin{align*}
y_{103}^+(z)&=2\pi[\sqrt{4}F_{-4}(z,1)+ 3\sqrt{2}F_{-2}(z,1) +3F_{-1}(z,1)]\\&-2\pi(\sqrt{4}B_{0,103}^+(1;-4)+3\sqrt{2}B_{0,103}^+(1;-2)+3B_{0,103}^+(1;-1))\\&=2\pi[2F_{-4}(z,1)+ 3\sqrt{2}F_{-2}(z,1) +3F_{-1}(z,1)]-\frac{57}{13}.
\end{align*}

The order of vanishing of $\Delta_{103}(z)$ at the cusp is $\nu_{103}=(12\cdot 52\pi/3)/(4\pi)=52$.  Therefore, for an arbitrary, fixed $w\in\HH$, the form
$$f_{103,w}(z)=\Delta_{103}^3(z)(x_{103}^+(z)-x_{103}^+(w))^{52}$$ is the weight $3k_{103}=36$ holomorphic form which
has constant term in the $q$-expansion equal to $1$.  Let $\{w_{1}, w_{2}, w_{3}\}$ be the three, not necessarily distinct,
points in the fundamental domain $\mathcal{F}$ where $(x_{103}^+(z)-x_{103}^+(w))$ vanishes. One of the points $w_{j}$ is equal to $w$. The form $f_{103,w_j}(z)$ vanishes at $z=w_{j}$ to order $52 \cdot\mathrm{ord}(w_j)$, $j=1,2,3$.

From Section \ref{sect Atkin Leh groups_KLF}, we have that
$$
\beta_{103}=\frac{3}{52\pi}\left(\frac{53}{104}\log103+2-2\log(4\pi)-24\zeta'(-1)\right)
$$
and $P_{103}(z)=\log\left(|\eta(z)\eta(103z)|^2\cdot \Im (z)\right)$.
Let us now apply Corollary \ref{cor:j-function} with $g(z)= x_{103}^+(z)$, in which case $c(g)=-15/13$.
In doing so, we get that

\begin{align*}
\langle x_{103}^+, \log (\Vert f_{103,w}\Vert ) \rangle &= - 208\pi^2
\sum\limits_{j=1}^{3}\left(\frac{\partial}{\partial s}\left. (\sqrt{3}F_{-3}(w_j,s) +F_{-1}(w_j,s))\right|_{s=1}\right)
\\& +120\pi\sum\limits_{j=1}^{3}\left(\log\left(|\eta(w_j)\eta(103w_j)|^2\cdot\Im (w_{j})\right)\right)
\\&-360\pi\left(\frac{53}{104}\log103-2\log(4\pi) -24\zeta'(-1)\right).
\end{align*}
Similarly, we can take $g(z)= y_{103}^+(z)$, in which case $c(g)=-57/13$ and we get that
\begin{align*}
\langle y_{103}^+, \log (\Vert f_{103,w}\Vert ) \rangle &= - 208\pi^2 \sum\limits_{j=1}^{3}
\left(\frac{\partial}{\partial s}\left. (2F_{-4}(w_{j},s)+ 3\sqrt{2}F_{-2}(w_{j},s) +3F_{-1}(w_{j},s))\right|_{s=1}\right)
\\& +456\pi\sum\limits_{j=1}^{3}\left(\log\left(|\eta(w_{j})\eta(103w_{j})|^2\cdot\Im (w_{j})\right)\right)
\\& -1368\pi\left(\frac{53}{104}\log103-2\log(4\pi) -24\zeta'(-1)\right).
\end{align*}

\subsection{An alternative formulation}
In the above discussion, we have written the constant $\beta$ and the Kronecker limit function $P$ separately.  However, it should be pointed out that in all instances the appearance of these terms are in the combination $\beta \vol_{\hyp}(M)- P(z)$.  From \eqref{KronLimitPArGen}, we can write
$$
\beta \vol_{\hyp}(M)- P(z)
= \frac{1}{\vol_{\hyp}(M)} \textrm{\rm CT}_{s=1} \E^{\mathrm{par}}_{\infty}(z,s),
$$
where $\textrm{\rm CT}_{s=1}$ denotes the constant term in the Laurent expansion at $s=1$.  It may
be possible that such notational change can provide additional insight concerning the formulas presented above.

\vspace{5mm}
\noindent

\noindent
James Cogdell \\
Department of Mathematics \\
Ohio State University \\
231 W. 18th Ave \\
Columbus, OH 43210,
U.S.A. \\
e-mail: cogdell@math.ohio-state.edu

\vspace{5mm}\noindent
Jay Jorgenson \\
Department of Mathematics \\
The City College of New York \\
Convent Avenue at 138th Street \\
New York, NY 10031
U.S.A. \\
e-mail: jjorgenson@mindspring.com

\vspace{5mm}\noindent
Lejla Smajlovi\'c \\
Department of Mathematics \\
University of Sarajevo\\
Zmaja od Bosne 35, 71 000 Sarajevo\\
Bosnia and Herzegovina\\
e-mail: lejlas@pmf.unsa.ba
\end{document}